\documentclass[11pt, reqno]{amsart}
\usepackage{amsmath, amsthm, amscd, amsfonts, amssymb, graphicx, color,   mathrsfs}
\usepackage[english]{babel}
    \usepackage[bookmarksnumbered, colorlinks, plainpages]{hyperref}
    \hypersetup{colorlinks=true,linkcolor=blue, anchorcolor=blue, citecolor=blue, urlcolor=red, filecolor=magenta, pdftoolbar=true}

\newtheorem{theorem}{Theorem}[section]
\newtheorem{lemma}[theorem]{Lemma}
\newtheorem{proposition}[theorem]{Proposition}
\newtheorem{corollary}[theorem]{Corollary}
\theoremstyle{definition}
\newtheorem{definition}[theorem]{Definition}

\theoremstyle{remark}
\newtheorem{remark}[theorem]{Remark}
\numberwithin{equation}{section}

\def\aa{{\mathcal A}}

\def\cc{{\mathcal C}}
\def\dd{{\mathcal D}}

\def\kk{{\mathcal K}}

\def\pp{{\mathcal P}}

\def\ss{{\mathcal S}}
\def\tt{{\mathcal T}}

\def\ie{i.e.\ }
\def\tq{\ :\ }

\def\eps{\varepsilon}
\def\dst{\displaystyle}

\def\eg{{\it e.g.\ }}

\def\tr{{\mathrm{tr}}}
\def\supp{{\operatorname{supp}}}

\def\card{{\operatorname{card}}}

\renewcommand{\Im}{\operatorname{Im}}

\def\C{{\mathbb{C}}}

\def\R{{\mathbb{R}}}
\def\S{{\mathbb{S}}}

\def\Z{{\mathbb{Z}}}

\def\d{\,{\mathrm{d}}}
\newcommand{\norm}[1]{{\left\|{#1}\right\|}}

\newcommand{\abs}[1]{{\left|{#1}\right|}}
\newcommand{\scal}[1]{{\left\langle{#1}\right\rangle}}

\date{\today}
\begin{document}

\title{Fourier-like multipliers and applications for integral operators}

\author{Saifallah Ghobber}

\address{LR11ES11 Analyse Math\'{e}matiques et Applications, Facult\'{e} des Sciences de Tunis, Universit\'{e} de Tunis El
Manar, 2092 Tunis, Tunisie}
\email{saifallah.ghobber@math.cnrs.fr}

\begin{abstract}
Timelimited functions and bandlimited functions play a fundamental role in signal and image processing.  But by the uncertainty principles, a signal cannot be simultaneously time and bandlimited. A natural assumption is thus that a signal is almost time and almost bandlimited. The aim of this paper is to  prove  that the set of almost time and almost bandlimited signals is not excluded from the uncertainty principles. The transforms under consideration are integral operators  with bounded kernels   for which there is a Parseval Theorem. Then we define the \emph{wavelet}  multipliers for this class of operators, and study their  boundedness and Schatten class properties. We show that the wavelet multiplier is unitary equivalent to a scalar multiple of the  phase space restriction operator. Moreover we prove that a signal which is almost time and almost bandlimited can be  approximated by its projection on the span of the first eigenfunctions of the  phase space restriction operator,   corresponding to the largest eigenvalues which are close to one.
 \end{abstract}



\maketitle

 \tableofcontents

\section{Introduction}
Timelimited functions and bandlimited functions are basic tools of signal and image processing. Unfortunately, the simplest form of the
uncertainty principle tells us that a signal cannot be simultaneously time  and bandlimited. This leads to the  investigation of the set of almost time and almost bandlimited functions, which has been initially carried through Landau, Pollak \cite{landau, LP} and then by Donoho, Stark \cite{ds}.
In the current paper, we review  the uncertainty principles on this set and present and compare different measure of localization. We made use of compositions of time and bandlimiting operators and considered the eigenvalue problem associated with these operators. The resulting operators yield an orthonormal set of eigenfunctions (well-known as prolate spheroidal functions) which satisfy some optimality in concentration in a   region in the time-frequency domain.
 We prove a characterization of functions that are approximately time and bandlimited in the region of interest, and we obtain approximation inequalities for such functions using a finite linear combination of eigenfunctions.

\bigskip

The aim of this paper is to continue the study of the uncertainty principle to a very general class of integral operators, which has been started in \cite{GJstudia, Gapplicable}. The transforms under consideration are integral operators $\tt$ with   bounded kernels $\kk$ and for which there is a Parseval Theorem. This class includes the usual Fourier transform, the Fourier-Bessel (or Hankel) transform, the  Dunkl transform and the deformed Fourier transform as particular cases.
A version of Hardy's and Donoho-Stark's  uncertainty principles for integral operators  has been proved in \cite{CDS, dejeu}.
In this paper, we consider results of a  different nature on  the subspaces of functions that are essentially timelimited on  $S$ and
bandlimited on $\Sigma$,  or  functions that are essentially concentrated on  $S$ and bandlimited on $\Sigma$, where $S$ and $\Sigma$ are   general subsets of finite measure.

\medskip

Let us now be more precise. Let $\Omega$ be a convex cones in $\R^d$ ({\it i.e.}
$\lambda x\in\Omega$ if $\lambda>0$ and $x\in\Omega$) with non-empty interior, endowed    with the Borel measure
$\mu$. The Lebesgue spaces $L^p(\Omega,\mu)$,  $1\leq p \le \infty$,
are then defined in the usual way, where $\|\cdot\|_{\infty}$ is the usual essential supremum norm and form $1\leq p < \infty$,
$$
\|f\|_{p,\mu}^p=\int_{\Omega} |f(x)|^p\d\mu(x).
$$

We assume that the measure $\mu$ is absolutely continuous with respect to the Lebesgue measure and has a polar
decomposition of the form $\mbox{d}\mu(r\zeta)=r^{2a-1}\,\mbox{d}r\,Q(\zeta)\,\mbox{d}\sigma(\zeta)$
where $\mbox{d}\sigma$ is the Lebesgue measure on the unit sphere $\S^{d-1}$ of $\R^d$  and $Q\in L^1(\S^{d-1},\mbox{d}\sigma)$, $Q\not=0$.
Then $\mu$ is homogeneous of degree $2a$ in the following sense:
for every continuous function $f$ with compact support in $\Omega$ and every $\lambda>0$,
\begin{equation}
\label{eqmesure}
\int_{\Omega} f\left(\frac{x}{\lambda}\right)\d \mu(x)= \lambda^{2a}\int_{\Omega} f(x)\d \mu(x).
\end{equation}
One can then define the integral operator
$\tt$ on $\ss(\Omega)$ by
\begin{equation}\label{defop}
\tt f(\xi)=\int_{\Omega}f(x)\kk(x,\xi)\d \mu(x), \quad \xi\in \Omega,
\end{equation}
where $\kk:\Omega\times \Omega\longrightarrow \C$ is a kernel such that:
\begin{enumerate}
	\item $\kk$ is continuous,
  \item $\kk$ is   bounded: $|\kk(x,\xi)|\le c_{\kk}$,
  \item $\kk$ is homogeneous: $\kk(\lambda x,\xi)= \kk(x,\lambda \xi)$.
\end{enumerate}
Then $\tt$ extends into a continuous operator from $L^1(\Omega,\mu)$ to the space of bounded continuous functions $\cc( \Omega)$, with
\begin{equation}\label{inftynormT}
  \norm{\tt f }_{\infty }\le c_\kk \norm{f}_{1,\mu } .
\end{equation}
Further, if we introduce the dilation operator  $\dd_{\lambda} $, $\lambda>0$ by:
$$
\dd_{\lambda} f(x) = \frac{1}{\lambda^{a}} f\left(\frac{x}{\lambda}\right).
$$
Then the homogeneity of $\kk$ implies
\begin{equation}\label{dt}
\tt \dd_{\lambda}= \dd_{\frac{1}{\lambda}}\tt.
\end{equation}

The integral operators under consideration will be assumed to satisfy some of
the following properties that are common for Fourier-like transforms:
\begin{enumerate}
\item $\tt$ has an {\bf Inversion Formula}: {\sl When both $f \in L^1(\Omega,\mu)$ and
$\tt f \in L^1( \Omega , \mu   )$ we have $f\in \cc (\Omega)$ and
for almost every $x\in \Omega$:}

\begin{equation}\label{inversionT}
f(x)=\int_{ \Omega} \tt f (\xi) \overline{\kk(x, \xi)} \d \mu (\xi).
\end{equation}

\item $\tt$ satisfies {\bf Parseval's Theorem}: {\sl for every $f,g\in\ss(\Omega)$,

\begin{equation}\label{parsevalT}
\scal{\tt f , \tt g }_{ \mu} =\scal{f, g}_{\mu} ,
\end{equation}
}
\end{enumerate}
where $\scal{\cdot, \cdot}_{\mu}$ is the  inner product  defined   on the Hilbert spaces  $L^2(\Omega,\mu)$ by
$$
\scal{f, g}_{\mu} =\int_{\Omega} f(x)\overline{g(x)}\d\mu(x).
$$
In particular, $\tt$ extends to an unitary transform from $L^2(\Omega,\mu)$ onto $L^2(\Omega,\mu)$,  such that
\begin{equation}\label{plancherelT}
\norm{\tt f }_{2,  \mu}=\norm{f}_{2,\mu} ,
\end{equation}
and for all $f\in  L^2(\Omega,\mu)$,
\begin{equation}\label{inversionT2}
\tt^{-1}f(\xi)=\overline{\tt \bar f(\xi) }   ,\quad \xi\in\Omega.
\end{equation}
This family of transforms includes for instance the Fourier transform (see \eg \cite{malin}), the Hankel transform (see \eg  \cite{Gjat}), the the Dunkl transform (see \eg \cite{Gbultin}), the $G$-transform (see \eg  \cite{TY}), the deformed Fourier transform (see \eg \cite{bensaid}),...

\medskip

The  inversion formula gives us back a signal $f$ via \eqref{inversionT}, and this is the basis for pseudo-differential operators on $\Omega$. Indeed if $\sigma$ be a suitable function on $\Omega$, then we define the pseudo-differential operator $F_\sigma$ by
$$
F_\sigma f(x)=\int_\Omega\sigma(\xi) \tt f (\xi)\overline{\kk(x,\xi)}\d\mu (\xi).
$$
Pseudo-differential operators $F_\sigma$ are known as the   multiplier, and in the case where $\sigma=\chi_A$ is a characteristic function, the operator $F_\sigma$ is known  as the frequency limiting operator on $\Omega$, we simply denote it by $F_A$.
Now if $\sigma$ is identically equal to $1$, then $F_\sigma :L^2 (\Omega,\mu)\to L^2(\Omega,\mu)$ is the identity operator in view of \eqref{inversionT2}.


\medskip

Our starting point is the following general form of Heisenberg--type uncertainty inequality (see \cite{GJstudia, Gapplicable}).
\begin{theorem}\ \label{th:HeisenbergTint}
Let $s,\;\beta>0$. Then
\begin{enumerate}
  \item  there exists a constant  $C=C(s,a,\beta)$  such that for all $f\in L^2(\Omega,\mu)$,
\begin{equation}\label{eqheisenbergTint}
 \||x|^{s}f \|^{\beta}_{2 ,\mu} \;\||\xi|^{\beta}\tt f \|^{s}_{2,\mu} \ge C  \norm{f}^{s+\beta}_{2,\mu },
\end{equation}
  \item there exists a constant  $C=C(s,a,\beta)$  such that for all $f\in L^1(\Omega,\mu)\cap L^2(\Omega,\mu)$,
  \begin{equation}\label{upl0012bisint}
  \||x|^{s}  f \|_{1,\mu}^{ a+\beta }  \;\||\xi|^\beta  \tt f \|_{2,\mu}^{a+ s }\ge  C \,\norm{f}_{1,\mu}^{ a+s}
 \,\norm{f}_{2,\mu}^{a+ \beta }.
\end{equation}
\end{enumerate}
\end{theorem}
The proof of  Inequality \eqref{upl0012bisint} can be obtained by combining a Nash-type inequality\cite[Proposition 2.2]{Gapplicable} and a Carlson-type inequality\cite[Proposition 2.3]{Gapplicable}, while the proof of Inequality
\eqref{eqheisenbergTint} can be obtained from either the Faris-type local uncertainty inequalities \cite[Theorem A]{GJstudia}, or from the fact that the   Benedicks-Amrein-Berthier uncertainty inequality \cite[Theorem B]{GJstudia}.

\medskip
Theorem \ref{th:HeisenbergTint} can be refined for  orthonormal  sequence in $L^2(\Omega,\mu)$. In particular   an orthonormal sequence in  $L^2(\Omega,\mu)$ cannot have uniform time-frequency localization. This is a consequence of the following Shapiro-type uncertainty principles.

\medskip
 \noindent{\bf Theorem A.}\
{\sl Let   $s>0$ and let $\left\{f_n \right\}_{n=1}^{\infty}$ be an orthonormal sequence in $L^2(\Omega,\mu)$.
 \begin{enumerate}
 \item  There exists a positive constant $C$ such that, for every  $N\ge 1$,
\begin{equation}\label{eqshapiroint}
    \dst \sum_{n=1}^N  \left(\big\||x|^{s}f_n\big\|^{2}_{2,\mu}+\big\||\xi|^{s}\tt f_n \big\|^{2}_{2,\mu}\right)
    \ge C \,  N^{1+\frac{s}{2a}}.
\end{equation}
\item If $\left\{f_n \right\}_{n=1}^{ \infty}$ is an orthonormal basis for $ L^2(\Omega,\mu)$, then
\begin{equation}\label{eqshapiroint2}
\sup_n \left(\big\||x|^{s}f_n\big\|_{2,\mu} \,\big\||\xi|^{s}\tt f_n \big\|_{ 2,\mu }\right)= \infty.
\end{equation}
 \end{enumerate} }
Notice that the homogeneity of the kernel $\kk$ plays a key role only in the proof of Inequality \eqref{eqshapiroint2}. Moreover
the proof of  Theorem A is inspired from the classical result for the Fourier transform in \cite{malin}, where the author prove that Inequality \eqref{eqshapiroint} is sharp  and  the equality cases  are attained for the sequence of Hermite functions, (see also \cite{jampow}). For the Hankel transform \cite{Gjat} this inequality is also optimal and the optimizers are the sequence of Laguerre functions.

\bigskip

One would like to  find nonzero functions $f\in L^2(\Omega,\mu)$, which are timelimited on a subset $S\subset \Omega$  (\ie $\supp f\subset S$)   and bandlimited  on a subset $\Sigma\subset \Omega$ (\ie $\supp \tt f\subset \Sigma$). Unfortunately, such functions do not exist, because if $f$ is time and bandlimited on   subsets of finite measure, then $f=0$ (see \cite{GJstudia}). As a result, it is natural to replaced the exact support by the essential support, and  to focus on functions that are  essentially  time and bandlimited to a bounded region like $S\times\Sigma$ in the time-frequency plane. To do this, 
we introduce the time limiting operator
$$
E_S f= \chi_S f,\quad f\in  L^1(\Omega,\mu)\cup  L^2(\Omega,\mu),
$$
and from \cite{AP, ds, Gapplicable} we recall  the following definition.
\begin{definition}\
Let $0\le\eps<1$ and let $ S,\Sigma\subset\Omega$. Then
\begin{enumerate}
  \item a nonzero function $f\in L^2(\Omega,\mu)$ is $\eps$-concentrated on  $S$  if $\|E_{S^c} f \|_{2,\mu} \le  \eps  \|f\|_{2,\mu} $,
    \item a nonzero function $f\in L^1(\Omega,\mu)$ is $\eps$-timelimited on  $S$  if $\|E_{S^c} f \|_{1,\mu} \le  \eps  \|f\|_{1,\mu} $,
  \item a nonzero function $f\in L^2(\Omega,\mu)$ is $\eps$-bandlimited on  $\Sigma$  if $\|F_{\Sigma^c} f \|_{2,\mu} \le  \eps \|f\|_{2,\mu} $,
 \item  a  nonzero function $f\in L^2(\Omega,\mu)$ is   $\eps$-localized with respect to an operator $L: L^2(\Omega,\mu)\to L^2(\Omega,\mu)$ if
  $\norm{Lf-f}_{2,\mu}   \le \eps \|f\|_{2,\mu} .$
\end{enumerate}
\end{definition}
Here   $A^c= \Omega\backslash A $ is the complement of $A$ in $ \Omega$.
 It is clear that, if $f$ is $\eps$-bandlimited on  $\Sigma$ then by Inequality  \eqref{plancherelT},    $\tt f$ is $\eps$-concentrated on  $\Sigma$.
 Notice also that, the $\eps$-concentration measure was  introduced in \cite{landau, LP, ds}, and the idea of $\eps$-localization has been recently introduced in \cite{AP}, which arises from the concept of pseudospectra of linear operators.

\medskip

 If $\eps=0$ in  the $\eps$-concentration measures, then  $ S$  and $\Sigma$ are respectively the exact support of $f$  and $\tt f$,
 moreover when $\eps \in( 0,1)$,  $ S$  and $\Sigma$   may be considered as the \emph{essential} support of $f$  and $\tt f $ respectively.
 On the other hand,   a function $f  \in L^2(\Omega,\mu)$   is $\eps$-localized with respect to an operator $L$ is an  eigenfunction  of $L$ corresponding to the eigenvalue $1$, if $\eps=0$, otherwise $f$ is called an  $ \eps $-approximated eigenfunction (or $\eps$-pseudoeigenfunction) of $L$  with pseudoeigenvalue $1$, see \cite{landau, TE}. 
 For example,  since on
  $L^2(\Omega,\mu)$, we have $F_{\Sigma }  +F_{\Sigma^c}=E_{S } +E_{S^c}=I $, then   a nonzero function $f\in L^2(\Omega,\mu)$ is $\eps$-concentrated on  $S$ (resp. $\eps$-bandlimited on  $\Sigma$),   if and only if,  $f$ is  $\eps$-localized with respect to $E_{S}$ (resp. $\eps$-localized with respect to $F_{\Sigma}$).

\medskip

Let  $\eps_1,\eps_2 \in( 0,1)$ and let   $ S$, $\Sigma$ two measurable subsets of $\Omega$ such that $0< \mu(S),\,\mu(\Sigma)<\infty$.
We denote by $ L^2(\eps_1,\eps_2, S, \Sigma)$ the subspace of $ L^2(\Omega,\mu)$ consisting of  functions that are $\eps_1$-concentrated on  $S$  and $\eps_2$-bandlimited on  $\Sigma$ (clearly $ L^2  (0,0,S,\Sigma)=\emptyset$). We denote also by $ L^1\cap L^2(\eps_1,\eps_2, S, \Sigma)$ the subspace of $L^1(\Omega,\mu)\cap L^2(\Omega,\mu)$ consisting of  functions that are $\eps_1$-timelimited on  $S$  and $\eps_2$-bandlimited on  $\Sigma$.

\medskip

As a first result, we can remark that the essential supports $S $ and $\Sigma$ cannot be too small, and this is a simple consequence of the following Donoho-Stark type uncertainty principle (see \cite{dejeu, GJstudia, Gapplicable}):
\begin{enumerate}
  \item  If $f\in  L^2  (\eps_1,\eps_2,S,\Sigma)$ such that $\eps_1^2+\eps_2^2<1$, then
  \begin{equation}\label{DS1int}
    \mu(S)\mu(\Sigma)\ge c_\kk^{-2}\left(1 -\sqrt{\eps_1^2+\eps_2^2}\right)^2.
  \end{equation}

  \item  If $f\in  L^1\cap L^2  (\eps_1,\eps_2,S,\Sigma)$, then
  \begin{equation}\label{DS2int}
    \mu(S)\mu(\Sigma)\ge c_\kk^{-2} (1 - \eps_1)^2(1 - \eps_2^2).
  \end{equation}
\end{enumerate}
The second Inequality \eqref{DS2int} is stronger than \eqref{DS1int}, since it is true for all $\eps_1,\eps_2\in (0,1)$, and  since in \eqref{DS2int} we can
give separately a lower bound for  $\mu(S)$ and $\mu(\Sigma)$, (see Inequality \eqref{eqdsoldl12}.

\medskip
It is natural to ask if there is a Heisenberg-type uncertainty inequalities  for functions in the subspaces $L^2  (\eps_1,\eps_2,S,\Sigma)$  and $L^1\cap L^2  (\eps_1,\eps_2,S,\Sigma)$, with constant   which depends on  $\eps_1,\eps_2,S,\Sigma$.
In Section 3, we use the local  uncertainty principles for functions either in $L^2  (\Omega,\mu)$  or in $L^1(\Omega,\mu)\cap L^2 (\Omega,\mu)$  to  obtain  an uncertainty inequalities comparing the support and the essential support with the time dispersion or the frequency dispersion.

\medskip
 \noindent{\bf Theorem B.}\
{\sl Let $\eps_1,\eps_2\in[0,1)$.
 \begin{enumerate}
 \item If $s,\beta>a$, then
 \begin{enumerate}
\item  there exists a constant $C$ such that for all function $f\in L^2(\Omega,\mu)$ which is $\eps_1$-concentrated on $ S$,
\begin{equation}\label{esssupdiseq2bisint}
\mu(S)^{\frac{\beta}{a}}  \,  \||\xi|^\beta \tt f \|_{2,\mu}^2  \ge  C  \left(1-\eps_1^2 \right)^{\frac{\beta}{a}} \|f\|_{2,\mu}^2 ,
\end{equation}

   \item     there exists a constant $C$ such that for all function $f\in L^2(\Omega,\mu)$ which is $\eps_2$-bandlimited on $ \Sigma$,
\begin{equation}\label{esssupdist2eq2int}
\mu(\Sigma)^{\frac{s}{a}}  \, \norm{|x|^s f}_{2,\mu}^2  \ge  C  \left(1-\eps_2^2 \right)^{\frac{s}{a}} \|f\|_{2,\mu}^2 .
\end{equation}
\end{enumerate}
\item If $s,\beta>0 $, then
\begin{enumerate}
   \item there exists a constant $C$ such that for all  function  $f\in L^1(\Omega,\mu)\cap L^2(\Omega,\mu)$, which is $\eps_1$-timelimited on $S$,
   \begin{equation}\label{eqdisnorml1int}
  \mu(S)^{\frac{a+\beta}{2a}} \,   \||\xi|^{\beta}\tt f\|_{2,\mu}\ge C\left(   1-\eps_1   \right)^{\frac{a+\beta}{a}}\norm{f}_{1,\mu},
 \end{equation}
   \item there exists a constant $C$ such that for all     function  $f\in L^1(\Omega,\mu)\cap L^2(\Omega,\mu)$, which is $\eps_2$-bandlimited on $\Sigma$,
   \begin{equation}\label{eqdisnorml1bisint}
\mu(\Sigma)^{\frac{a+s }{2a}} \,   \norm{|x|^{s}f}_{1,\mu}\ge C\left( 1-\eps_2^2  \right)^{\frac{a+s }{2a}}   \norm{f}_{2,\mu}.
  \end{equation}
  \end{enumerate}
     \end{enumerate} }
Of course, if $\eps_1=\eps_2=0$, then $S=\supp f$ and $\Sigma=\supp \tt f$. Combining the inequalities in Theorem A, we obtain the following Heisenberg-type uncertainty principles, which can be viewed as the $\eps$-concentration version of   Theorem \ref{th:HeisenbergTint}:
\begin{enumerate}
  \item If $s,\beta>a$, then there exists a constant $C$ such that for all $f\in L^2  (\eps_1,\eps_2,S,\Sigma)$,
  \begin{equation}\label{heisnew2int}
  \norm{|x|^s f}_{2,\mu}^{\beta}\,\||\xi|^\beta \tt f\|_{2,\mu}^{s} \ge C \left(\frac{(1-\eps_1^2)(1-\eps_2^2)}{\mu(S)\mu(\Sigma)}  \right)^{\frac{s\beta}{2a}} \norm{f}_{2,\mu}^{s+\beta}.
   \end{equation}
  \item  If $s,\beta>0$, then there exists a constant $C$ such that for all    $f\in  L^1\cap L^2(\eps_1,\eps_2, S, \Sigma)$,
   \begin{equation}\label{eqdisnorml12}
 \norm{|x|^{s}f}_{1,\mu}^{a+\beta}\,\||\xi|^{\beta}\tt f\|_{2,\mu}^{a+s}\ge C
 \left(\frac{(1-\eps_1)^2 (1-\eps_2^2)}{\mu(S)\mu(\Sigma)}\right)^{\frac{(a+s)(a+\beta)}{2a}} \norm{f}_{1,\mu}^{a+s} \norm{f}_{2,\mu}^{a+\beta}.
 \end{equation}
\end{enumerate}
Notice that Inequalities \eqref{esssupdiseq2bisint}, \eqref{esssupdist2eq2int} and  \eqref{heisnew2int} hold  also for $0<s,\beta \le a$, but not necessarily with the same constants. Notice also that  results in   Theorem B are stronger than Inequalities \eqref{heisnew2int} and \eqref{eqdisnorml12}, since in Theorem B we have a lower bounds for the measures of the time and frequency dispersions  separately, this give  more information than a lower bound of the product between them.

\bigskip

Now let $\phi$ and $\psi$ two bounded functions in  $L^2(\Omega,\mu)$ such that $\|\phi\|_{2,\mu}=\|\psi\|_{2,\mu}$ and $\|\phi \|_\infty  \|\psi \|_\infty=1$ . The first aim of Section 4 is to make precise the definition of the pseudo-differential operator (known as the wavelet multiplier)  $\bar\psi F_\sigma\phi:   L^2(\Omega,\mu)\to L^2(\Omega,\mu)$, where $\sigma$ is a symbol in $ L^p(\Omega,\mu)$, $1\le p\le \infty$,  and to prove that the resulting bounded linear operator is in the Schatten-von Neumann class $S_p$. On the other hand, we use the $\eps$-localization measure introduced  in \cite{AP} to state a new uncertainty inequality involving the wavelet multiplier. More precisely we establish the following results.
\medskip

  \noindent{\bf Theorem C.}\
{\sl Let $\sigma\in L^p(\Omega,\mu)$, $1\le p\le \infty$, and let $\eps_1,\eps_2\in(0,1)$ such that $\eps_1+\eps_2<1$.
  \begin{enumerate}
    \item The wavelet multiplier $\bar\psi F_\sigma\phi :L^2(\Omega,\mu)\to L^2(\Omega,\mu)$ is in $S_p$ and
\begin{equation}\label{eq.estmh3}
    \|\bar\psi F_\sigma\phi \|_{S_p}\le  c_\kk^{\frac{2}{p}} \|\sigma \|_{p,\mu}  .
\end{equation}
    \item If a nonzero function $f \in L^2(\Omega,\mu)$ is  $\eps_1$-localized with respect to  $(\bar\psi F_S\phi)$ and  $\eps_2$-localized with respect to   $(\bar\psi F_\Sigma\phi)$, then
\begin{equation}
  \mu(S)\mu(\Sigma)\ge  c_\kk^{-4}(1-\eps_1-\eps_2).
\end{equation}
  \end{enumerate}}
Here we denote by  $S_\infty $ for the space of bounded operators from $L^2(\Omega,\mu)$ into itself. Notice also that the related results of Theorem B and  Theorem C for the Hankel transform has been studied by the author in \cite{Gbanach}.

\bigskip

In Section 5, we will prove that the wavelet multiplier is unitary equivalent to   a scalar multiple of  the  phase space restriction operator $L_{S,\Sigma}=E_SF_\Sigma S_S$ on $ L^2  (\Omega,\mu)$ arising from the Landau-Pollak  theory in signal analysis  \cite{landau, LP}. 
This leads to a compact self-adjoint operator with spectral representation:
\begin{equation}
L_{S,\Sigma} f=\sum_{n=1}^{\infty} \lambda_n \scal{f,\varphi_n}_{\mu}\varphi_n .
\end{equation}
In the classical setting, these eigenfunctions are known as the prolate spheroidal wave functions. In particular,
\begin{equation}
\|L_{S,\Sigma}\varphi_n-\varphi_n \|_{2,\mu}=\scal{\varphi_n-L_{S,\Sigma}\varphi_n,\varphi_n}_{\mu}=1-\lambda_n<1,
\end{equation}
then each eigenfunction $\varphi_n$ is $(1-\lambda_n)$-localized  with respect to $L_{S,\Sigma}$, and a simple computation shows also that each  function $f$ in  $ L^2(\eps_1,\eps_2,S,\Sigma)$  is $(2\eps_1+\eps_2)$-localized  with respect to $L_{S,\Sigma}$.
 Moreover, if a function $f\in L^2(\Omega,\mu)$ is $\eps$-localized with respect to $L_{S,\Sigma}$, then it satisfies
\begin{equation}
   \scal{f-L_{S,\Sigma}f,f}_\mu\le 2\eps  \|f\|_{2,\mu}^2.
\end{equation}
Conversely, if  we denote by
\begin{equation}
L^2(\eps,S,\Sigma)=\left\{f\in L^2(\Omega,\mu) \tq   \scal{f-L_{S,\Sigma}f,f}_\mu\le  \eps  \|f\|_{2,\mu}^2 \right\},
\end{equation}
then, each  function $f\in L^2(\eps,S,\Sigma)$ is $\sqrt{\eps}$-localized with respect to $L_{S,\Sigma}$, and each $f\in L^2(\eps_1,\eps_2,S,\Sigma)$ is in
$ L^2(2\eps_1+\eps_2,S,\Sigma)$ (see Proposition \ref{prop.comparison} for more details).

\medskip

Now let $n(\eps,S,\Sigma)=\card \{n\tq \lambda_n\ge1-\eps\}$ the number of eigenvalues which are close to one.
 In \cite{LP}, Landau and Pollak gave an asymptotic estimate for $n(\eps,S,\Sigma)$, when $\tt$ is the Fourier transform and $S,\Sigma$ are real intervals. This result can be interpreted as follows:  there exist, up to a small error, $\frac{|S||\Sigma|}{2\pi}$ independent functions 
   that are $\eps_1$-concentrated on $S$ and $\eps_2$-bandlimited on $\Sigma$, these functions are the so-called prolate spheroidal wave functions. The last estimate has been recently refined in \cite{AP}, where the authors instead of counting the number of eigenfunctions with eigenvalue
close to one, they count the maximum number of orthogonal functions 
 that are $\eps$-localized with respect to $L_{S,\Sigma}$. 
In \cite{AP2}, Abreu and Pereira noted that the sharp asymptotic number of these orthogonal functions is $ \approx  (1-\eps)^{-1}\frac{|S||\Sigma|}{2\pi}$.
On the other hand, we establish the following results, which characterize functions that are in $ L^2(\eps,S,\Sigma)$, and approximate almost time and bandlimited functions.

\medskip
\noindent{\bf Theorem D.}\
{\sl Let $f_{\mathrm{ker}}$ denote the orthogonal projection of $f$ onto the kernel   of $L_{S,\Sigma}$.
  \begin{enumerate}
    \item   A function $f $ is in  $ L^2(\eps,S,\Sigma)$ if and only if,
\begin{equation*}\  
    \sum_{n=1}^{n(\eps,S,\Sigma)}(\lambda_n +\eps-1) \abs {\scal{   f,\varphi_n }_{\mu}}^{2} \ge (1 - \eps)\|f_{\mathrm{ker}}\|_{2,\mu}^{2}+
     \sum_{n= 1+n(\eps,S,\Sigma)}^{\infty}(1-\eps-\lambda_n) \abs {\scal{f,\varphi_n }_{\mu }}^{2}.
\end{equation*}
    \item  If $f$ is  in $L^2(\eps,S,\Sigma) $, then
\begin{equation}\label{eq.aproximtion}
  \norm{f-\sum_{n=1}^{n(\eps_0,S,\Sigma)}\scal{f,\varphi_n}_\mu\varphi_n}_{2,\mu}\le \sqrt{\frac{\eps}{\eps_0}}\;\|f\|_{2,\mu}.
\end{equation}
  \end{enumerate}}
Using the above comparison, it follows that,
if $f\in L^2(\eps_1,\eps_2,S,\Sigma) $, then
\begin{equation}
  \norm{f-\sum_{n=1}^{n(\eps_0,S,\Sigma)}\scal{f,\varphi_n}_\mu\varphi_n}_{2,\mu}\le \sqrt{\frac{2\eps_1+\eps_2}{\eps_0}}\;\|f\|_{2,\mu},
\end{equation}
and if $f\in L^2(\Omega,\mu) $ is $\eps$-localized with respect to $L_{S,\Sigma}$, then
\begin{equation}
  \norm{f-\sum_{n=1}^{n(\eps_0,S,\Sigma)}\scal{f,\varphi_n}_\mu\varphi_n}_{2,\mu}\le \sqrt{\frac{2\eps}{\eps_0}}\;\|f\|_{2,\mu}.
\end{equation}


\section{Preliminaries}
\subsection{Notation} Throughout this paper we denote by $\scal{\cdot,\cdot}$
the usual Euclidean inner product in $\R^d$, we write for $x \in \R^d$,
$\abs{x}=\sqrt{\scal{x,x}}$  and if $A$ is a measurable subset in $\R^d$, we will write $A^c$ for its  complement in  $\Omega$.

For $\xi\in \Omega$, we denote by $\kk_\xi:\Omega\to\Omega$ the kernel defined by $\kk_\xi(x)=\kk(x,\xi)$, and for $r > 0$, we denote by $B_r$   the closed ball in $\Omega$ centred at $0 $ and of radius $ r$.

We will write along this paper $C$ for a constant that depends on the parameters $a$, $s$ and $c_\kk$
defined above (and may be depends also on some other parameter $\beta $, $\eps$, $\ldots$).
This constant  may changes from line to line.


\subsection{Generalities}
Let X be a separable and complex Hilbert space (of infinite dimension) in which the inner product
and the norm are denoted by $\scal{\cdot, \cdot}$ and $\|\cdot\|$ respectively. Let $\aa:X\to X $ be a compact operator for which  we denote by $\aa^*:X\to X $ its adjoint. Then the linear operator $|\aa|=\sqrt{\aa^*\aa}:X\to X $ is   positive and compact.
The singular values 
of   $\aa$ are the eigenvalues of the   self-adjoint operator $|\aa| $.
  For $1\le p< \infty$, the Schatten class $S_p$ is the space of all compact operators whose singular values lie in $\ell_p$. 
   In particular, $S_2 $ is the space of Hilbert-Schmidt operators, and $S_1$ is the space of trace class operators.
   Moreover  from \cite[Section VI.6]{ReedSimon}  and \cite[Proposition 2.6]{wong},  we have the following criterion for a bounded linear  operator to be in the trace class.
  \begin{proposition}\ \label{wongcompacts1}
Let  $\aa:X\to X $ be a bounded linear operator such that, for all orthonormal bases $\{\varphi_n\}_{n=1}^\infty$ for $X$,
   \begin{equation}\label{compfini}
    \sum_{n=1}^\infty \abs{\scal{\aa \varphi_n,\varphi_n}} <\infty,
\end{equation}
 then   $\aa:X \to X $ is in the trace class $S_1$ with,
\begin{equation} \label{eq:traceS1}
 \tr(\aa)=\sum_{n=1}^\infty \scal{\aa \varphi_n,\varphi_n} ,
\end{equation}
where $\{\varphi_n\}_{n=1}^\infty$ is any  orthonormal basis for $X$.
\end{proposition}
If, in addition $\aa$ is positive,
 then (see \cite[Proposition 2.7]{wong}),
\begin{equation}\label{eqtrpositive}
\norm{\aa}_{S_1}=  \tr(\aa).
\end{equation}
Moreover from \cite[Proposition 2.8]{wong}, we have the
 following   criterion for a bounded linear operator $\aa:X \to X $ to be in the Hilbert-Schmidt class $S_2$.
\begin{proposition}\ \label{wongHS}
Let  $\aa:X\to X $ be a bounded linear operator such that, for all orthonormal bases $\{\varphi_n\}_{n=1}^\infty$ for $X$,
   \begin{equation}\label{compfiniHS}
    \sum_{n=1}^\infty  \|\aa \varphi_n\|^2 <\infty,
\end{equation}
 then   $\aa:X \to X $ is in the Hilbert-Schmidt class   $S_2$ with,
\begin{equation} \label{eq:S2HS}
\|\aa\|_{S_2}^2= \sum_{n=1}^\infty  \|\aa \varphi_n\|^2 ,
\end{equation}
where $\{\varphi_n\}_{n=1}^\infty$ is any  orthonormal basis for $X$.
\end{proposition}

 Finally, if the compact operator  $\aa: X \to X$  is Hilbert-Schmidt, then the positive operator $\aa^*\aa$ is in the space of trace class $S_1$ and
\begin{equation} \label{eq:traceS2}
\norm{\aa}_{HS}^2   :=\norm{\aa}_{S_2}^2 =\norm{\aa^*\aa}_{S_1} =\tr(\aa^*\aa)=\sum_{n=1}^\infty \norm{\aa \varphi_n}^2,
\end{equation}
for any orthonormal basis $\{\varphi_n\}_{n=1}^\infty$  for $X$.

\smallskip

For consistency, we define $S_\infty:=B\left(X\right)$ to be the space of bounded operators from $X$ into $X$, equipped with norm,
   \begin{equation}
  \norm{\aa}_{S_\infty} = \sup_{f\tq \norm{f}\le1}  \norm{\aa f}.
   \end{equation}
It is obvious that $S_p \subseteq S_q$, $1\le p \le q\le \infty$.

\subsection{Fourier-like Multipliers}
For $\sigma\in    L^\infty(\Omega,\mu)$, we define the linear operator  $F_\sigma:L^2(\Omega,\mu)\to L^2(\Omega,\mu)$ by
 \begin{equation}\label{LO}
   F_\sigma f=  \tt^{-1}\left[ \sigma \, \tt f \right].
 \end{equation}
In the case of the Fourier transform, this operator is known as the Fourier multiplier. Clearly, if  $\sigma=1$, then $F_\sigma = I$, where $I$ is the identity operator. Moreover, from the formula \eqref{plancherelT}, it is clair that $F_\sigma$ is bounded with
\begin{equation}
   \| F_\sigma  \|_{S_\infty}\le \|\sigma  \|_{\infty} .
\end{equation}

\begin{definition}\
Let $\sigma\in  L^1(\Omega,\mu) \cup L^\infty(\Omega,\mu)$ and let $\phi,\psi\in  L^\infty(\Omega,\mu)\cap L^2(\Omega,\mu)$ such that $\|\phi\|_{2,\mu}=\|\psi\|_{2,\mu}=1 $. We define the linear operator $P_{\sigma,\phi,\psi}:L^2(\Omega,\mu)\to L^2(\Omega,\mu)$ by
\begin{equation}\label{locoper2}
\scal{P_{\sigma,\phi,\psi} f, g}_{\mu}= \scal{\sigma \tt(\phi f) , \tt(\psi g)}_{\mu}.
\end{equation}
\end{definition}
In the case of the Fourier transform, this operator is known as the \emph{wavelet} Fourier multiplier, which can be viewed as a variant of a localization operator with respect to the symbol $\sigma$ and the admissible wavelets $\phi$ and $\psi$, see   the book \cite{wong} and the reference therein.
Notice that, if $\sigma=\chi_A$ is the characteristic function on the subset $A \subset \Omega$, then we  write $F_\sigma$  as $F_A$. In this case, we also write  $P_{\sigma,\phi,\psi}$ as  $P_{A,\phi,\psi}$ if   $\phi\neq\psi$ and $P_{A,\phi}$, if   $\phi=\psi$. The linear operator $F_A:L^2(\Omega,\mu)\to L^2(\Omega,\mu)$ is a self-adjoint projection, it is known as  the \emph{frequency limiting operator} on $L^2(\Omega,\mu)$ and has many applications in time-frequency analysis. Moreover we will prove in the last section that $P_{A,\phi}$ can be viewed as the \emph{phase space} (or \emph{time frequency}) \emph{limiting operator}.

\medskip

The next proposition shows that  $P_{\sigma,\phi,\psi}:L^2(\Omega,\mu)\to L^2(\Omega,\mu)$ and $\bar\psi F_\sigma \phi:L^2(\Omega,\mu)\to L^2(\Omega,\mu)$ are unitary equivalent.
\begin{proposition}\
Let $\sigma\in  L^1(\Omega,\mu) \cup L^\infty(\Omega,\mu)$ and let $\phi,\psi\in  L^\infty(\Omega,\mu)\cap L^2(\Omega,\mu)$ such that $\|\phi\|_{2,\mu}=\|\psi\|_{2,\mu}=1 $. Then
\begin{equation}
\scal{P_{\sigma,\phi, \psi} f, g}_{\mu}= \scal{\bar\psi F_\sigma \phi , g}_{\mu}.
\end{equation}
\end{proposition}
 \begin{proof}
From \eqref{LO} and Parseval's formula \eqref{parsevalT}, we have
 \begin{eqnarray*}
  \scal{P_{\sigma,\phi,\psi} f, g}_{\mu} &=& \scal{\sigma \tt(\phi f) , \tt(\psi g)}_{\mu} \\
   &=& \scal{\tt  F_\sigma (\phi f)  , \tt(\psi g)}_{\mu}\\
     &=&  \scal{ F_\sigma (\phi f)  ,  \psi g}_{\mu} \\
     &=&   \scal{(\bar \psi F_\sigma  \phi) f   ,   g}_{\mu} .
 \end{eqnarray*}
 The proof is complete.
 \end{proof}

\section{Uncertainty principles by means of the frequency limiting operator}
 We will need to introduce the following  time limiting   operator, defined by
\begin{equation*}
E_S f= \chi_S f , \quad \quad f\in   L^1(\Omega,\mu)\cup L^2(\Omega,\mu),
\end{equation*}
where  $S\subset\Omega$. 
Clearly $E_S : L^2(\Omega,\mu)\to  L^2(\Omega,\mu)$ is a self-adjoint   projection.

\subsection{Uncertainty principles on the space $ L^2(\eps_1,\eps_2, S, \Sigma)$}
The first known result  for functions in $ L^2(\eps_1,\eps_2, S, \Sigma)$ is the following Donoho-Stark type uncertainty inequality, see \cite[Inequality (3.4)]{GJstudia}.
\begin{theorem}\label{thdsold}
Let $\eps_1,\eps_2 \in( 0,1)$ such that $\eps_1^2+\eps_2^2<1$. Then if $ f\in L^2(\eps_1,\eps_2, S, \Sigma)$ we have
\begin{equation}\label{eqdsold}
    \mu(S)\mu(\Sigma)\ge c_\kk^{-2}\left( 1-\sqrt{\eps_1^2+\eps_2^2} \right)^2.
\end{equation}
\end{theorem}
 In the case of the Fourier transform, it goes back to Donoho and Stark \cite{ds}. This inequality implies that the essential support of $f$ and $\tt f$ cannot be too small. Moreover, we recall the following local uncertainty principle, see \cite{GJstudia}.
 \begin{theorem}\ \label{localT1}
\begin{enumerate}
  \item If $\,0 <s <a$, then there exists a constant $C$ such that for all $f\in L^2(\Omega,\mu)$ and all measurable subset $\Sigma \subset \Omega$  of finite measure $0<\mu(\Sigma)<\infty$,

 \begin{equation}\label{localT1eq1}
\norm{F_\Sigma f}_{2, \mu  }^2 \leq
C\,\mu(\Sigma)^{\frac{s}{ a}} \,\|\abs{x}^s f \|_{2,\mu}^2.
\end{equation}
\item If $\,s >a$, then there exists a constant $C$ such that for all $f\in L^2(\Omega,\mu)$ and all measurable subset $\Sigma \subset \Omega$  of finite measure $0<\mu(\Sigma)<\infty$,
\begin{equation}\label{localT1eq2}
\norm{F_\Sigma f}_{2, \mu }^2 \leq
C \,\mu(\Sigma) \,\|f \|_{2,\mu}^{2-\frac{2a}{s}} \|\abs{x}^s f \|_{2,\mu}^{\frac{2a}{s}}.
\end{equation}
\end{enumerate}
 \end{theorem}
 Next, take  $s=a$. Then, if we apply the first inequality \eqref{localT1eq1} with $ a(1-\eps)$, $\eps\in(0,1)$, replacing $s$    and then apply the following classical inequality
 \begin{equation}\label{eq:classineq}
 \||x|^{ a-a\eps} f \|_{2,\mu }\leq C \norm{f}_{2,\mu }^{\eps}
\norm{|x|^a f}_{2,\mu }^{1-  \eps },
\end{equation}
we obtain for all $\eps\in(0,1)$,
\begin{equation}\label{localT1eq3}
\norm{F_\Sigma f}_{2, \mu}^2 \leq
C \,\mu(\Sigma)^{1- \eps } \,\|f \|_{2,\mu}^{  2\eps } \,\|\abs{x}^a f \|_{2,\mu}^{2- 2\eps }.
\end{equation}
Consequently we conclude  the following first corollary  comparing the support of   $\tt f$  and the generalized time  dispersion $\||x|^s f\|_{2,\mu}$ for function in the range of $F_\Sigma$:
$$
\Im(F_\Sigma)=\{f\in L^2(\Omega,\mu)\tq \supp \tt f\subset \Sigma    \}.
$$
\begin{corollary}\ \label{cor1}
Let $s>0$. Then there exists a constant $C$ such that for all $f\in \Im(F_\Sigma)$,

 \begin{equation}\label{localT1eq4}
 \mu\left(\supp \tt f\right) \|\abs{x}^s f \|_{2,\mu}^{\frac{2a}{ s}}\ge C \norm{f}_{2,\mu}^{\frac{2a}{ s}}.
\end{equation}
\end{corollary}
\begin{proof}
Let $s>0$ and  $f\in \Im(F_\Sigma)$. Then $f=F_\Sigma f$, and   we apply \eqref{localT1eq1}, \eqref{localT1eq2}, \eqref{localT1eq3} to obtain the desired result.
  \end{proof}
Notice that, if $\mu\left(\supp \tt f\right)$ is finite, then $\mu\left(\supp   f\right)$ is infinite, because $f$ and $\tt f$ cannot be simultaneously   supported on subsets of finite measure, see \cite[Corollary 3.7]{GJstudia}. This result is known as the Benedicks-Amrein-Berthier uncertainty principle.

Moreover, we can also obtain   an inequality comparing the essential support of   $\tt f$  and the generalized time dispersion $\||x|^s f\|_{2,\mu}$ for functions that are $\eps_2$-bandlimited on $ \Sigma$.

\begin{corollary}\ \label{cor2}
Let $s>0$.
\begin{enumerate}
  \item If $\,0 <s <a$, then there exists a constant $C$ such that for     all function $f$ which is $\eps_2$-bandlimited on $ \Sigma$,
\begin{equation}\label{esssupdisteq1}
 \mu(\Sigma)^{\frac{s}{a}}   \, \norm{|x|^s f}_{2,\mu}^2 \ge  C  \left(1-\eps_2^2\right)   \|f\|_{2,\mu}^2 .
\end{equation}
\item If $ s >a$, then there exists a constant $C$ such that for all function $f$ which is $\eps_2$-bandlimited on $ \Sigma$,
\begin{equation}\label{esssupdist2eq2}
\mu(\Sigma)^{\frac{s}{a}}  \, \norm{|x|^s f}_{2,\mu}^2  \ge  C  \left(1-\eps_2^2 \right)^{\frac{s}{a}} \|f\|_{2,\mu}^2 .
\end{equation}
\item For all $\eps\in(0,1)$, there exists a constant $C$ such that for all function $f$ which is $\eps_2$-bandlimited on $ \Sigma$,
\begin{equation}\label{esssupdisteq3}
 \mu(\Sigma) \,\norm{|x|^a f}_{2,\mu}^2  \ge  C (1-\eps_2^2)^{\frac{1}{ 1-\eps}}   \|f\|_{2,\mu}^2 .
\end{equation}
\end{enumerate}
\end{corollary}

 \begin{proof}
Since  $f\in  L^2(\Omega,\mu)$ is $\eps_2$-bandlimited on $\Sigma$, then

$$
\|F_{\Sigma} f \|_{2,\mu}^2=  \|f\|_{2,\mu}^2 - \|F_{\Sigma^c} f \|_{2,\mu}^2\ge (1-\eps_2^2) \|f\|_{2,\mu}^2.
$$
For the first result, we use the local inequalities \eqref{localT1eq1}.
Analogously, for the second inequality, we  use \eqref{localT1eq2}, and
finally,  for the third  inequality, we  use \eqref{localT1eq3}.
\end{proof}

\bigskip

Now, since, $\norm{F_\Sigma f}_{2, \mu} = \norm{E_\Sigma \tt f}_{2, \mu}$, then by interchanging the roles of $f$ and $\tt f$
  in Theorem \ref{localT1}, Corollary \ref{cor1} and Corollary \ref{cor2} , we obtain the following   results involving the time limiting operator instead of the frequency limiting operator, and the frequency dispersion instead of the time dispersion.

\begin{theorem}\ \label{localT2}
Let $\beta >0$.
\begin{enumerate}
  \item If $\,0 <\beta <a$, then
   \begin{enumerate}
  \item there exists a constant $C$ such that for all $f\in L^2(\Omega,\mu)$ and all measurable subset $S \subset \Omega $  of finite measure $0< \mu (S)<\infty$,
\begin{equation}\label{localT2eq1}
\norm{E_S f}_{2, \mu}^2 \leq C
  \mu (S) ^{\frac{\beta}{ a}} \|\abs{x}^\beta \tt f\big\|_{2,\mu}^2,
\end{equation}

\item there exists a constant $C$ such that for     all function $f$ which is $\eps_1$-concentrated on $ S$,
\begin{equation}\label{esssupdiseq1bis}
 \mu(S)^{\frac{\beta}{a}}   \,  \||\xi|^\beta \tt f \|_{2,\mu}^2 \ge  C  \left(1-\eps_1^2\right)   \|f\|_{2,\mu}^2.
\end{equation}
\end{enumerate}
\item If $ \beta >a$,  then
\begin{enumerate}
\item   there exists a constant $C$ such that for all $f\in L^2(\Omega,\mu)$ and all measurable subset $S \subset \Omega $  of finite measure $0< \mu(S)<\infty$,
\begin{equation}\label{localT2eq2}
\norm{E_S f}_{2, \mu}^2 \leq
C \, \mu(S)  \|f \|_{2,\mu}^{2-\frac{2a}{\beta}} \|\abs{\xi}^\beta \tt f \|_{2, \mu}^{\frac{2a}{\beta}},
\end{equation}

\item there exists a constant $C$ such that for all function $f$ which is $\eps_1$-concentrated on $ S$,
\begin{equation}\label{esssupdiseq2bis}
\mu(S)^{\frac{\beta}{a}}  \,  \||\xi|^\beta \tt f \|_{2,\mu}^2  \ge  C  \left(1-\eps_1^2 \right)^{\frac{\beta}{a}} \|f\|_{2,\mu}^2 .
\end{equation}
\end{enumerate}
\item For all $\eps\in(0,1)$,
\begin{enumerate}
\item there exists a constant $C$ such that for all $f\in L^2(\Omega,\mu)$ and all measurable subset $S \subset \Omega $  of finite measure $0< \mu(S)<\infty$,
\begin{equation}\label{localT2eq3}
\norm{E_S f}_{2, \mu}^2 \leq
C \,\mu(S)^{1- \eps } \,\|f \|_{2,\mu}^{2\eps } \,\|\abs{\xi}^a \tt f \|_{2,\mu}^{2- 2\eps },
\end{equation}

\item there exists a constant $C$ such that for all function $f$ which is $\eps_1$-concentrated on $ S$,
\begin{equation}\label{esssupdiseq3bis}
 \mu(S) \,\norm{|\xi|^a \tt f}_{2,\mu}^2  \ge  C \left(1-\eps_1^2\right)^{\frac{1}{ 1-\eps}}   \|f\|_{2,\mu}^2 .
\end{equation}
\end{enumerate}
\item There exists a constant $C$ such that for all $f\in \Im(E_S)=\{f\in L^2(\Omega,\mu)\tq \supp f\subset S\}$,
\begin{equation}\label{localT2eq4}
 \mu\left(\supp  f\right) \|\abs{\xi}^\beta \tt f \|_{2,\mu}^{\frac{2a}{ \beta}}\ge C\norm{f}_{2,\mu}^{\frac{2a}{ \beta}}.
\end{equation}
\end{enumerate}
  \end{theorem}
Finally we can formulate our new Heisenberg-type uncertainty inequalities for functions in $ L^2(\eps_1,\eps_2, S, \Sigma)$, with constants that depend on $\eps_1,\,\eps_2,\, S$ and $ \Sigma$.
 \begin{theorem}\ \label{thheisenbergnew}
Let $s,\beta>0$. Then for all $f\in  L^2(\eps_1,\eps_2, S,\Sigma)$:
\begin{enumerate}
  \item if $0<s,\beta<a$,
   \begin{equation}\label{heisnew1}
    \norm{|x|^s f}_{2,\mu}^{\beta}\, \||\xi|^\beta \tt f \|_{2,\mu}^{s} \ge  C \;\frac{ (1-\eps_1^2)^{s/2}(1-\eps_2^2)^{\beta/2} }{  \left(\mu(S)\mu(\Sigma)\right)^{\frac{s\beta}{2a}}}\norm{f}_{2,\mu}^{s+\beta},
      \end{equation}
       \item if $s,\beta>a$,
  \begin{equation}\label{heisnew2}
  \norm{|x|^s f}_{2,\mu}^{\beta}\,\||\xi|^\beta \tt f\|_{2,\mu}^{s} \ge C \left(\frac{(1-\eps_1^2)(1-\eps_2^2)}{\mu(S)\mu(\Sigma)}  \right)^{\frac{s\beta}{2a}} \norm{f}_{2,\mu}^{s+\beta},
   \end{equation}
      \item for all $\eps\in(0,1)$,
      \begin{equation}\label{heisnew3}
    \norm{|x|^{a} f}_{2,\mu} \, \norm{|\xi|^{a} \tt f}_{2,\mu} \ge C\; \frac{\big( (1-\eps_1^2) (1-\eps_2^2)\big)^{ \frac{1}{2-2\eps}}}{\sqrt{ \mu(S)\mu(\Sigma)}} \norm{f}_{2,\mu}^{2}.
      \end{equation}
\end{enumerate}
 \end{theorem}

\begin{remark}\
 \begin{enumerate}
 \item Notice that Corollary \ref{cor2} and Inequalities \eqref{esssupdiseq1bis}, \eqref{esssupdiseq2bis} and \eqref{esssupdiseq3bis}  give separately  a lower bounds for the measures of the time dispersion $ \norm{|x|^s f}_{2,\mu}$ and the frequency dispersion $\||\xi|^\beta \tt f\|_{2,\mu} $, which  give  more  information than a lower bound of the product
between them in Theorem \ref{thheisenbergnew} 

\item
On the other hand,  from Corollary \ref{cor2} and Inequalities \eqref{esssupdiseq1bis}, \eqref{esssupdiseq2bis} and \eqref{esssupdiseq3bis}, we can obtain  separately  a lower bounds, that depend of the signal $f\in  L^2(\eps_1,\eps_2, S,\Sigma)$, for the measures of $\mu (S)$ and $\mu (\Sigma)$, from which we deduce, in the spirit of \cite{BCO}, the following lower bounds for the product between them:
 \begin{equation}\label{dsnew}
    \mu(S)\mu(\Sigma)\ge
  \begin{cases}
C.C_f (s,a,\beta) \big( (1-\eps_1^2)^{\frac{1}{\beta}}(1-\eps_2^2)^{\frac{1}{s}}\big)^{ a},   & 0<s,\beta<a,\\\\
C.C_f(s,a,\beta) \, (1-\eps_1^2)(1-\eps_2^2),& s,\beta>a,\\\\
C.C_f(a,a, a)   \big((1-\eps_1^2) (1-\eps_2^2) \big)^{ \frac{1}{1-\eps}},   & \mathrm{otherwise},
       \end{cases}
  \end{equation}
   where  $C$ is a constant that depend only on $s,a,c_\kk,\beta ,\eps,$ and
  \begin{equation}\label{eqconstantf}
       C_f(s,a,\beta)=\left(  \frac{\norm{f}_{2,\mu}^{s+\beta}
      }{ \norm{x^s f}_{2,\mu}^{\beta} \||\xi|^\beta \tt f\|_{2,\mu}^s}\right)^{\frac{2a}{s\beta}}.
    \end{equation}
    \end{enumerate}
    \end{remark}

 \subsection{Uncertainty principles on the space $ L^1\cap L^2(\eps_1,\eps_2, S, \Sigma)$}
The first known result  for functions in $L^1\cap L^2(\eps_1,\eps_2, S, \Sigma)$ is the following Donoho-Stark type uncertainty inequality, see \cite[Proposition 2.6]{Gapplicable}.
\begin{theorem}\label{thdsoldl12}
Let $\eps_1,\eps_2 \in( 0,1)$. Then if $ f\in L^1\cap  L^2(\eps_1,\eps_2, S, \Sigma)$ we have
\begin{equation}\label{eqdsoldl12}
    \mu(S)\ge \frac{\norm{f}_{1,\mu}^2}{\norm{f}_{2,\mu}^2}  (1- \eps_1)^2,  \quad \mu(\Sigma)\ge c_\kk^{-2} \frac{\norm{f}_{2,\mu}^2}{\norm{f}_{1,\mu}^2}  (1- \eps_2^2),
\end{equation}
and then
\begin{equation}\label{eqdsoldl122}
    \mu(S)\mu(\Sigma)\ge c_\kk^{-2} (1- \eps_1)^2(1- \eps_2^2).
\end{equation}
\end{theorem}
 Theorem \ref{thdsoldl12} is stronger then Theorem \ref{thdsold}, in the sense that the previous theorem give a lower bound of $ \mu(S)$ and $\mu(\Sigma)$ separately, which is not possible in Theorem \ref{thdsold}.

 Now we will recall the following Carlson-type and Nash-type inequalities, see \cite[Proposition 2.2, Proposition 2.3]{Gapplicable}.
 \begin{theorem}\ \label{thcarlsonnash}
Let $s,\beta>0$. Then we have:
 \begin{enumerate}

   \item \emph{ A Carlson-type inequality:}
  there exists a constant $C=C(s,a)$ such that for all    $f\in L^1(\Omega,\mu)\cap L^2(\Omega,\mu)$,
\begin{equation}\label{carlson}
  \norm{f}_{1,\mu}^{1+\frac{s}{a}}\le C   \|f\|_{2,\mu}^{\frac{s}{a}} \;\norm{|x|^{s}f}_{1,\mu}.
\end{equation}

 \item \emph{A Nash-type inequality:}  there exists a constant $C=C (\beta,a)$ such that for all    $f\in L^1(\Omega,\mu)\cap L^2(\Omega,\mu)$,
\begin{equation}\label{nash}
  \norm{f}_{2,\mu}^{1+\frac{\beta}{a}}\le C  \|f\|_{1,\mu}^{\frac{\beta}{a}} \;\||\xi|^\beta \tt f\|_{2,\mu}.
\end{equation}
   \end{enumerate}
\end{theorem}
Consequently we obtain a lower bounds for the time and frequency dispersions:
\begin{equation} \label{eqtwodispersions}
     \norm{|x|^{s}f}_{1,\mu}\ge C \left(\frac{ \norm{f}_{1,\mu} }{\|f\|_{2,\mu}} \right)^{\frac{s}{a}}\norm{f}_{1,\mu}\quad \mathrm{and} \quad
     \||\xi|^\beta \tt f\|_{2,\mu}\ge C \left(\frac{ \norm{f}_{2,\mu} }{\|f\|_{1,\mu}}  \right)^{\frac{\beta}{a}}\norm{f}_{2,\mu}.
\end{equation}
  \begin{corollary}\ Let $s,\beta>0$. Then

  \begin{enumerate}
    \item   there exists a constant $ C=C(a,\beta, s)$ such that for all    $f\in L^1(\Omega,\mu)\cap L^2(\Omega,\mu)$,
\begin{equation}\label{upl0012}
  \||x|^{s}  f \|_{1,\mu}^{ \beta }  \;\||\xi|^\beta  \tt f \|_{2,\mu}^{ s }\ge  C \,\norm{f}_{1,\mu}^{ \beta}
 \,\norm{f}_{2,\mu}^{ s },
\end{equation}
 \item  there exists a constant $C$ such that for all $f\in L^1(\Omega,\mu)\cap L^2(\Omega,\mu)$ and all measurable subset of $\Sigma$ of finite measure,
\begin{equation}\label{FSlocal12}
 \norm{F_\Sigma f}_{2,\mu}^2  \le   C \, \mu(\Sigma) \norm{f}^{\frac{2s}{a+s}}_{2,\mu}\,
  \norm{|x|^{s}f}^{\frac{2a}{a+s}}_{1,\mu},
\end{equation}
    \item   there exists a constant $  C $ such that for all $f\in L^1(\Omega,\mu)\cap L^2(\Omega,\mu)$ and all measurable subset   $S$ of finite measure,
\begin{equation}\label{ESlocal12}
\norm{E_S f}_{1,\mu}^2  \le   C\, \mu(S) \norm{f}^{\frac{2\beta}{ a+\beta}}_{1,\mu}\,
\||\xi|^{\beta}\tt f\|^{\frac{2a}{a+\beta }}_{2,\mu},
\end{equation}
\item  there exists a constant $  C $ such that for all $f\in L^1(\Omega,\mu)\cap L^2(\Omega,\mu)$ with $\supp \tt f\subset\Sigma$,
  \begin{equation}\label{FSlocal12supp}
 \mu\big(\supp \tt f \big) \,  \norm{|x|^{s}f}^{\frac{2a}{a+s}}_{1,\mu} \ge    C \, \norm{f}^{\frac{2a}{a+s}}_{2,\mu},
\end{equation}
  \item  there exists a constant $  C $ such that  for all $f\in L^1(\Omega,\mu)\cap L^2(\Omega,\mu)$ with $\supp f\subset S$,
  \begin{equation}\label{ESlocal12supp}
 \mu\big(\supp f\big) \,  \||\xi|^{\beta}\tt f\|^{\frac{2a}{a+\beta}}_{2,\mu} \ge   C  \, \norm{f}^{\frac{2a}{a+\beta}}_{1,\mu}.
\end{equation}
  \end{enumerate}
\end{corollary}

\begin{proof}
The first inequality follows by combining the Carlson inequality \eqref{carlson} and the Nash inequality \eqref{nash}. Next by  \eqref{plancherelT} and \eqref{inftynormT},
$$
\norm{F_\Sigma f}_{2,\mu}^2 =\norm{\chi_\Sigma\tt f }_{2,\mu}^2\le \mu(\Sigma)  \norm{\tt f}_{\infty}^2
\le c_\kk^2  \mu(\Sigma)  \norm{f}_{1,\mu}^2,
$$
and by the Carlson inequality \eqref{carlson} we obtain \eqref{FSlocal12}.
 Now by the Cauchy-Schwartz inequality we have,
$$
\norm{E_S f}_{1,\mu}^2  \le \mu(S)  \norm{f}_{2,\mu}^2,
$$
and by the Nash type inequality \eqref{nash} we deduce \eqref{ESlocal12}. Finally \eqref{FSlocal12supp} follows directly from \eqref{FSlocal12} by taking $\Sigma= \supp \tt f$ and if we take $S=\supp f$ in \eqref{ESlocal12} we obtain \eqref{ESlocal12supp}.
\end{proof}
\begin{remark}\
Clearly, Inequality  \eqref{eqtwodispersions} implies also that, for all   $f\in L^1(\Omega,\mu)\cap L^2(\Omega,\mu)$,
\begin{equation}\label{upl0012bis}
  \||x|^{s}  f \|_{1,\mu}^{ a+\beta }  \;\||\xi|^\beta  \tt f \|_{2,\mu}^{a+ s }\ge  C \,\norm{f}_{1,\mu}^{ a+s}
 \,\norm{f}_{2,\mu}^{a+ \beta }.
\end{equation}
\end{remark}
 \begin{corollary}\
 Let $s,\beta>0$. Then
 \begin{enumerate}
   \item there exists a constant $C$ such that for all  function  $f$, which is $\eps_1$-timelimited on $S$,
   \begin{equation}\label{eqdisnorml1}
  \mu(S)^{\frac{a+\beta}{2a}} \,   \||\xi|^{\beta}\tt f\|_{2,\mu}\ge C\left(   1-\eps_1   \right)^{\frac{a+\beta}{a}}\norm{f}_{1,\mu},
 \end{equation}
   \item there exists a constant $C$ such that for all     function  $f$, which is $\eps_2$-bandlimited on $\Sigma$,
   \begin{equation}\label{eqdisnorml1bis}
\mu(\Sigma)^{\frac{a+s }{2a}} \,   \norm{|x|^{s}f}_{1,\mu}\ge C\left( 1-\eps_2^2  \right)^{\frac{a+s }{2a}}   \norm{f}_{2,\mu},
  \end{equation}
   \item  there exists a constant $C$ such that for all    $f\in  L^1\cap L^2(\eps_1,\eps_2, S, \Sigma)$,
   \begin{equation}\label{eqdisnorml12}
 \norm{|x|^{s}f}_{1,\mu}^{a+\beta}\,\||\xi|^{\beta}\tt f\|_{2,\mu}^{a+s}\ge C
 \left(\frac{(1-\eps_1)^2 (1-\eps_2^2)}{\mu(S)\mu(\Sigma)}\right)^{\frac{(a+s)(a+\beta)}{2a}} \norm{f}_{1,\mu}^{a+s} \norm{f}_{2,\mu}^{a+\beta}.
 \end{equation}
 \end{enumerate}
\end{corollary}

\begin{proof}
If  $f$ is $\eps_1$-timelimited, then
$$
\norm{E_S f}_{1,\mu}\ge \norm{ f}_{1,\mu} -\norm{E_{S^c} f}_{1,\mu}\ge(1-\eps_1)\norm{ f}_{1,\mu},
$$
and if  $f$ is $\eps_2$-bandelimited, then
$$
\norm{F_\Sigma f}_{2,\mu}^2=\norm{ f}_{2,\mu}^2-\norm{F_{\Sigma^c} f}_{2,\mu}^2\ge(1-\eps_2^2)\norm{ f}_{2,\mu}^2.
$$
Hence the desired result follows from \eqref{FSlocal12} and  \eqref{ESlocal12}.
\end{proof}
\begin{remark}\
 Let $s,\beta>0$ and let   $f\in  L^1(\Omega, \mu)\cap L^2(\Omega, \mu)$.
 \begin{enumerate}
   \item If $f$ is $\eps_1$-timelimited on $S$, then
    \begin{equation}\label{eqdisnorml1ds}
   \mu(S)\ge C\left(\frac{\norm{f}_{1,\mu}}{\||\xi|^{\beta}\tt f\|_{2,\mu}}\right)^{\frac{2a }{a +\beta}}(1-\eps_1)^2.
 \end{equation}
   \item   If $f$ is $\eps_2$-bandlimited on $\Sigma$, then
    \begin{equation}\label{eqdisnorml1ds2}
   \mu(\Sigma)\ge C\left(\frac{\norm{f}_{2,\mu}}{\norm{|x|^{s}f}_{1,\mu}}\right)^{\frac{2a }{a +s}} (1-\eps_2^2).
  \end{equation}
   \item If $f\in  L^1\cap L^2(\eps_1,\eps_2, S, \Sigma)$, then
   \begin{equation}\label{eqdisnorml12ds}
  \mu(S)\mu(\Sigma)\ge C. \tilde C_f(a,s,\beta)
  (1-\eps_1)^2(1-\eps_2^2),
 \end{equation}
where
 \begin{equation}
 \tilde C_f(a,s,\beta)=\left(\frac{\norm{f}_{1,\mu}^{a+s} \norm{f}_{2,\mu}^{a+\beta}}{\norm{|x|^{s}f}_{1,\mu}^{a+\beta}  \||\xi|^{\beta}\tt f\|_{2,\mu}^{a+s}}   \right)^{\frac{2a}{(a+\beta)(a+s)}}.
  \end{equation}
 \end{enumerate}
\end{remark}

\section{The wavelet   multiplier}
 Our motivation here came from the classical setting stated in \cite{WZ1, WZ2}.
 In this section let $\phi$ and $\psi$ will be two functions in $  L^\infty(\Omega,\mu)\cap L^2(\Omega,\mu)$ such that $\|\phi\|_{2,\mu}=\|\psi\|_{2,\mu}=1 $.
 \subsection{Boundedness}
The aim of this section is to prove that we can also define $P_{\sigma,\phi,  \psi} $ for symbol $\sigma\in   L^p(\Omega,\mu)$, $1<p<\infty$.
First, if $\sigma\in L^\infty(\Omega,\mu)$, we have the following result.
\begin{proposition}\ \label{pr.estmh1}
Let $\sigma\in L^\infty(\Omega,\mu)$. Then $P_{\sigma,\phi,  \psi} $ is in $S_\infty $ and
\begin{equation}\label{eq.estmh1}
    \|P_{\sigma,\phi,  \psi} \|_{S_\infty}\le  \|\phi \|_\infty \|\psi \|_\infty  \|\sigma \|_\infty.
\end{equation}
\end{proposition}
\begin{proof}
By the Cauchy-Schwartz inequality,
$$
|\scal{P_{\sigma,\phi,  \psi} f, g }_{\mu}|\le  \|\sigma \|_\infty   \| \tt(\phi f) \|_{2,\mu}\| \tt(\psi g) \|_{2,\mu}.
$$
Then by Plancherel's formula \eqref{plancherelT}, we obtain
\begin{eqnarray*}
  |\scal{P_{\sigma,\phi,  \psi} f, g }_{\mu}| &\le& \|\sigma \|_\infty   \| \phi f  \|_{2,\mu}\| \psi g  \|_{2,\mu} \\
   &\le&   \|\sigma \|_\infty   \|\phi \|_\infty \|\psi \|_\infty\|  f  \|_{2,\mu}\|  g  \|_{2,\mu}.
\end{eqnarray*}
This completes the proof.
\end{proof}

Now, if we consider $\sigma\in L^1(\Omega,\mu)$, then we obtain the following result.
\begin{proposition}\ \label{pr.estmh2}
Let $\sigma\in L^1(\Omega,\mu)$. Then $P_{\sigma,\phi,  \psi} $ is in $S_\infty $ and
\begin{equation}\label{eq.estmh2}
    \|P_{\sigma,\phi,  \psi} \|_{S_\infty}\le c_\kk^2 \|\sigma \|_{1,\mu}  .
\end{equation}
\end{proposition}
\begin{proof}
Since $ \tt(\phi f)(\xi)= \scal{f, \overline{\phi \,\kk_\xi } }_{\mu}$, then by the Cauchy-Schwartz inequality,
$$
\| \tt(\phi f)\|_\infty\le c_\kk \|f\|_{2,\mu}\|\phi\|_{2,\mu}.
$$
Therefore, since $\|\phi\|_{2,\mu}=\|\psi\|_{2,\mu} =1$, we obtain
\begin{eqnarray*}
  |\scal{P_{\sigma,\phi,  \psi} f, g }_{\mu}| &\le& \|\sigma \|_{1,\mu}   \|\tt(\phi f) \|_{\infty}\| \tt(\psi g)  \|_{\infty} \\
   &\le& c_\kk^2  \|\sigma \|_{1,\mu}    \|f\|_{2,\mu}\|g\|_{2,\mu}.
\end{eqnarray*}
This completes the proof.
\end{proof}
Thus, by \eqref{eq.estmh1}, \eqref{eq.estmh2}  and the  Riesz-Thorin interpolation argument \cite[Theorem 2]{stein}  (see also \cite[Theorem 12.4]{wong}) we obtain the following theorem.
\begin{theorem}\ \label{th.estmh3}
Let $\sigma\in L^p(\Omega,\mu)$, $1<p<\infty$. Then 
the linear operator $P_{\sigma,\phi,  \psi} :L^2(\Omega,\mu)\to L^2(\Omega,\mu)$ is bounded and
\begin{equation}\label{eq.estmh3}
    \|P_{\sigma,\phi,  \psi} \|_{S_\infty}\le c_\kk^{\frac{2}{p}}   \|\phi \|_\infty^{\frac{1}{p'}} \|\psi \|_\infty^{\frac{1}{p'}} \|\sigma \|_{p,\mu}  .
\end{equation}
\end{theorem}
Hence we can define the operator $(\bar \psi F_\sigma  \phi) :L^2(\Omega,\mu)\to L^2(\Omega,\mu)$, where $\sigma\in L^p(\Omega,\mu)$, $1\le p\le \infty$ by
\begin{equation}\label{eqtwoop}
    \scal{P_{\sigma,\phi,\psi} f, g}_{\mu} =\scal{(\bar \psi F_\sigma  \phi) f   ,   g}_{\mu} .
\end{equation}

 \subsection{Schatten class properties}
Let us begin with the following theorem.
\begin{theorem}\ \label{th:estim3n}
  Let $\sigma$ be symbol in $L^1(\Omega,\mu)$. Then  $P_{\sigma,\phi,\psi} $ is Hilbert Schmidt and
\begin{equation}\label{estim3n}
   \|P_{\sigma,\phi,\psi}\|_{S_2}^2= \int_\Omega  \sigma(\xi) \scal{  P_{\bar\sigma,\psi,\phi} \, \bar\psi  ,   \bar\phi\,  |\kk_\xi|^2}_{\mu}\d \mu(\xi)\le  \|\sigma\|_{ L^1_\alpha}^2.
\end{equation}
  \end{theorem}

  \begin{proof}\
  First by \eqref{locoper2} it follows immediately that the adjoint of $P_{\sigma,\phi,\psi}$ is   $P_{\bar\sigma,\psi,\phi}: L^2(\Omega,\mu) \to L^2(\Omega,\mu)$. Now, 
  Let $\{\varphi_n\}_{n=1}^\infty$ be an orthonormal basis for  $ L^2(\Omega,\mu)$. Then by \eqref{locoper2} and by Fubini's
theorem we obtain,
  \begin{eqnarray*}
   \sum_{n=1}^\infty \|P_{\sigma,\phi,\psi}\varphi_n\|^2_{2,\mu}   &=&  \sum_{n=1}^\infty \scal{P_{\sigma,\phi,\psi}\varphi_n,P_{\sigma,\phi,\psi}\varphi_n}_{\mu} \\
    &=&  \sum_{n=1}^\infty \scal{\sigma \tt(\phi\, \varphi_n),   \tt(\psi\, P_{\sigma,\phi,\psi}\varphi_n)}_{\mu}\\
      &=&  \sum_{n=1}^\infty\int_\Omega  \sigma(\xi)  \scal{\varphi_n,  \overline{\phi\, \kk_\xi}\,}_{\mu}
       \overline{\scal{P_{\sigma,\phi,\psi}\varphi_n,  \overline{\psi \,\kk_\xi}\,}_{\mu}} \d \mu(\xi) \\
      &=&   \int_\Omega  \sigma(\xi) \sum_{n=1}^\infty  \scal{ P_{\bar\sigma,\psi,\phi}   \overline{\psi \kk_\xi},\varphi_n}_{\mu}
      \scal{\varphi_n,  \overline{\phi \kk_\xi}}_{\mu}\d \mu(\xi)
      \\
      &=&    \int_\Omega  \sigma(\xi) \scal{  P_{\bar\sigma,\psi,\phi}   \overline{ \psi \kk_\xi},    \overline{\phi \kk_\xi}\,}_{\mu}\d \mu(\xi),
  \end{eqnarray*}
  where we have used   Parseval's identity in the last line. Therefore  from Proposition \ref{pr.estmh2} and since $|\kk_\xi|\le c_\kk$,
  \begin{eqnarray*}
    \sum_{n=1}^\infty \|P_{\sigma,\phi,\psi}\varphi_n\|^2_{2,\mu} &\le & \|P_{\bar\sigma,\psi,\phi} \|_{S_\infty} \|\phi\|_{2,\mu}\|\psi\|_{2,\mu} \|\sigma\|_{1,\mu}\\
      &\le& c_\kk^4  \|\sigma\|_{1,\mu}^2.
  \end{eqnarray*}
  Thus from  Proposition \ref{wongHS}, the operator $P_{\sigma,\phi,\psi}$ is in $S_2$  and $\|P_{\sigma,\phi,\psi}\|_{S_2}\le c_\kk^2  \|\sigma\|_{1,\mu}.$
  \end{proof}
Consequently the operator  $P_{\sigma,\phi,\psi}$ is also compact for symbols in $L^p(\Omega,\mu)$.
\begin{corollary}\ \label{pr:compact}
Let $\sigma$ be symbol in $L^p(\Omega,\mu)$, $1\le p< \infty$. Then the operator $P_{\sigma,\phi,\psi}$ is compact.
\end{corollary}
\begin{proof}
Let $\{\sigma_n\}_{n=1}^\infty$ be a sequence of functions in $L^1(\Omega,\mu)\cap L^\infty (\Omega,\mu)$  such that $\sigma_n \to \sigma$  in $L^p(\Omega,\mu)$ as $n\to \infty$. Then by Theorem \ref{th.estmh3},
\begin{equation}
     \|P_{\sigma_n,\phi,\psi}-P_{\sigma,\phi,\psi}\|_{S_\infty}\le \|\phi \|_\infty^{\frac{1}{p'}} \|\psi \|_\infty^{\frac{1}{p'}} \|\sigma_n-\sigma \|_{p,\mu}   .
\end{equation}
Therefore $P_{\sigma_n,\phi,\psi}\to P_{\sigma,\phi,\psi}$ in $S_\infty$ as $n\to \infty$. Now, since by Theorem \ref{th:estim3n}, the operators  $P_{\sigma_n,\phi,\psi}$ are in $S_2$ and hence compact, and since the set of compact operators is a closed  subspace of  $S_\infty$,
then the operator $P_{\sigma ,\phi,\psi}$ is also compact.
 \end{proof}
 More precisely we will prove that  the operator  $P_{\sigma ,\phi,\psi}$ is in fact in the Schatten class $S_p$,  $1\le p< \infty$. Of particular interest is the Schatten-von Neumann class $S_1$ (see \cite{WZ1, WZ2}).

\begin{theorem}\ \label{th.traceclass}
Let $\sigma\in L^1(\Omega,\mu)$. Then   $P_{\sigma,\phi,  \psi} :L^2(\Omega,\mu)\to L^2(\Omega,\mu)$ is  trace class with
\begin{equation}\label{eq.traceclass}
      \|P_{\sigma,\phi,  \psi} \|_{S_1}\le c_\kk^2  \|\sigma \|_{1,\mu},
\end{equation}
and we have the following trace formula
\begin{equation}\label{trace}
     \tr\left(P_{\sigma,\phi,  \psi}\right)= \int_\Omega \sigma(\xi) \scal{\overline{ \psi\,\kk_\xi}, \overline{ \phi\,\kk_\xi}\,}_{\mu}\d\mu(\xi).
\end{equation}
\end{theorem}

\begin{proof} Let $\{\varphi_n\}_{n=1}^\infty$ be an orthonormal basis for $L^2(\Omega,\mu)$. Then
\begin{eqnarray*}
 \sum_{n=1}^\infty  \scal{P_{\sigma,\phi,\psi} \varphi_n, \varphi_n}_{\mu}
   &= &  \sum_{n=1}^\infty \int_\Omega  \sigma(\xi)   \tt(\phi \,\varphi_n)(\xi) \overline{\tt(\psi \,\varphi_n)(\xi) }  \d\mu(\xi)\\
  &= &  \sum_{n=1}^\infty \int_\Omega  \sigma(\xi)
   \scal{ \overline{\psi\,\kk_\xi}, \varphi_n}_{\mu}   \scal{ \varphi_n,  \overline{\phi \,\kk_\xi}\,}_{\mu}  \d\mu(\xi).
\end{eqnarray*}
Thus by  Fubini's theorem,
\begin{equation}\label{eqequalitys1tr}
     \sum_{n=1}^\infty  \scal{P_{\sigma,\phi,\psi} \varphi_n, \varphi_n}_{\mu}=\int_\Omega  \sigma(\xi) \sum_{n=1}^\infty
   \scal{\overline{\psi\,\kk_\xi}, \varphi_n}_{\mu}   \scal{ \varphi_n, \overline{\phi \,\kk_\xi}\,}_{\mu}  \d\mu(\xi).
\end{equation}
Therefore by   Parseval's identity,  and the fact that $\|\phi  \|_{2,\mu}=\|\psi  \|_{2,\mu}  =1$,
\begin{eqnarray*}
 \sum_{n=1}^\infty \left|  \scal{P_{\sigma,\phi,\psi} \varphi_n, \varphi_n}_{\mu} \right|
    &\le&  \frac{1}{2}  \int_\Omega |\sigma(\xi)| \sum_{n=1}^\infty\left( \left|\scal{\overline{\phi \,\kk_\xi}, \varphi_n}_{\mu} \right|^2
    +\left|\scal{\overline{\psi \,\kk_\xi}, \varphi_n}_{\mu} \right|^2  \right)\d\mu(\xi)\\
    &=&  \frac{1}{2}  \int_\Omega |\sigma(\xi)| \left(  \|\phi \kk_\xi\|_{2,\mu}^2+\|\psi \kk_\xi \|_{2,\mu}^2 \right)\d\mu(\xi)\\
     &\le& c_\kk^2  \|\sigma \|_{1,\mu}.
\end{eqnarray*}
By Proposition \ref{wongcompacts1}, the operator $P_{\sigma,\phi,\psi}$ is in $S_1$ and with \eqref{eqequalitys1tr} and Parseval's identity,
$$
\tr(P_{\sigma,\phi,\psi})=  \sum_{n=1}^\infty  \scal{P_{\sigma,\phi,\psi} \varphi_n, \varphi_n}_{\mu}=
\int_\Omega \sigma(\xi) \scal{ \overline{\psi \kk_\xi},  \overline{\phi\,\kk_\xi}\,}_{\mu}\d\mu(\xi).
$$
This allows to conclude.
\end{proof}
 Moreover by \eqref{eq.estmh1}, \eqref{eq.traceclass} and by interpolation argument we deduce the following result.
\begin{corollary}\ \label{cr.sp}
Let $\sigma\in L^p(\Omega,\mu)$, $1<p<\infty$. Then   $P_{\sigma,\phi,  \psi} :L^2(\Omega,\mu)\to L^2(\Omega,\mu)$  is in $S_p$ and
\begin{equation}\label{eq.sp}
    \|P_{\sigma,\phi,  \psi} \|_{S_p}\le c_\kk^{2/p}  \|\phi \|_\infty^{\frac{1}{p'}} \|\psi \|_\infty^{\frac{1}{p'}} \|\sigma \|_{p,\mu}  .
\end{equation}
\end{corollary}

\subsection{An uncertainty relation}
In this subsection, we will assume that $\phi$ and $\psi$  satisfy  $\| \phi\|_{\infty}\|\psi\|_{\infty}= 1$.
Now let $ \sigma_1=\chi_S$ and $ \sigma_2=\chi_\Sigma$ and
 let $L_1=P_{\sigma_1,\phi,\psi}$ and $L_2=P_{\sigma_2,\phi,\psi}$.

\begin{theorem}\ \label{thdsloc}
Let   $\eps_1,\eps_2\in(0,1)$ such that $ \eps_1 + \eps_2 <1$. If $f\in L^2(\Omega,\mu)$ is $\eps_1$-localized with respect to $P_{\sigma_1,\phi,\psi}$ and $\eps_2$-localized with respect to $P_{\sigma_2,\phi,\psi}$ then,
 \begin{equation}
 \mu(S) \mu(\Sigma)\ge c_\kk^{-4}   \left( 1-  \eps_1 - \eps_2  \right) .
\end{equation}
\end{theorem}
\begin{proof}
From Proposition \ref{pr.estmh1},
 \begin{eqnarray*}
   \| f-L_2L_1 f\|_{2,\mu} &\le&  \| f-L_2  f\|_{2,\mu} +\| L_2f-L_2L_1 f\|_{2,\mu}  \\
     &\le&   \| L_2f-  f\|_{2,\mu} +\|L_2\|_{S_\infty}\| L_1f-  f\|_{2,\mu}\\
     &\le&   (\eps_2 + \eps_1 )\|f\|_{2,\mu}.
 \end{eqnarray*}
 Therefore
 \begin{eqnarray*}
   \| L_2L_1 f\|_{2,\mu} &\ge&  \| f \|_{2,\mu} -\| f-L_2L_1 f\|_{2,\mu}  \\
     &\ge&    (1- \eps_1 - \eps_2)  \|  f\|_{2,\mu} .
 \end{eqnarray*}
Thus from Proposition  \ref{pr.estmh2} it follows that
\begin{eqnarray*}
 1- \eps_1 -\eps_2       &\le&  \| L_2L_1  \|_{S_\infty} \\
 &\le&  \|L_1 \|_{S_\infty}\| L_2 \|_{S_\infty} \\
    &\le&  c_\kk^4\, \mu(S) \mu(\Sigma).
\end{eqnarray*}
This proves the desired result.
\end{proof}

Notice that, when $\eps_1 = \eps_2=0$ in the classical Donoho-Stark uncertainty inequality \eqref{eqdsold}, we have $\mu(S) \mu(\Sigma)\ge c_\kk^{-2}$. This is a trivial assertion since in this case $S=\supp f$, $\Sigma=\supp \tt f$ and then from \cite{GJstudia}, either $\mu(\supp f) =\infty$ or  $\mu(\supp \tt f) =\infty$.
The case $\eps_1 = \eps_2=0$  in Theorem \ref{thdsloc} is not trivial and gives the following result.

\begin{corollary}
If $f\in L^2(\Omega,\mu)$ is an eigenfunction of $P_{S,\phi,\psi}$ and   $P_{\Sigma,\phi,\psi}$  corresponding to the same eigenvalue $1$, then
 \begin{equation}
 \mu(S) \mu(\Sigma)\ge c_\kk^{-4}    .
\end{equation}

\end{corollary}
  \section{Uncertainty principles for orthonormal sequences}
\subsection{The phase space restriction operator}
We define the phase space restriction operator   by
$$
L_{S,\Sigma} =E_S F_\Sigma  E_S = ( F_\Sigma E_S)^*  F_\Sigma E_S.
$$
We know that the operator $ F_\Sigma E_S$ is Hilbert-Schmidt (see \cite[Inequality (3.2)]{GJstudia}), and since the pair $(S,\Sigma)$ is strongly annihilating, then from \cite{GJstudia}, we have
\begin{equation}
\|L_{S,\Sigma} \|_{S_\infty}=\|E_S F_\Sigma\|^2_{S_\infty}=\|F_\Sigma E_S \|^2_{S_\infty}<1.
\end{equation}
Moreover, the operator $ L_{S,\Sigma}$ is self-adjoint, positive and from \eqref{eq:traceS2}  it  is compact and even  trace class with
 \begin{equation}\label{eq:s1fsophs}
 \|L_{S,\Sigma} \|_{S_1}=\| F_\Sigma E_S\|_{S_2}^2,  
\end{equation}

In the fundamental paper  \cite{LP},   Landau and Pollak have considered the eigenvalue problem associated with the
positive self-adjoint operator $E_S F_\Sigma  E_S :   L^2(\R) \to  L^2(\R)$, where $S, \Sigma$ are real intervals, for which they   proved an asymptotic estimate for the number of eigenvalues.

Motivated by the  process in \cite{HW}, we will show that the phase space restriction operator $L_{S,\Sigma}$
can be viewed as a wavelet multiplier, and then we will deduce a trace formula.

  \begin{theorem}\ \label{FSRO}
 Let $\phi=\psi$ be the function on $\Omega$ defined by
 $
 \phi = \frac{1}{\sqrt{\mu(S)}}\chi_S
 $
 and let $\sigma=\chi_\Sigma$. Then
 \begin{equation}\label{FSRObb}
L_{S,\Sigma}= \mu(S) \, P_{\Sigma,\phi}.
 \end{equation}
 \end{theorem}
  \begin{proof}
  Clearly, the function $\phi$ belongs to  $L^2(\Omega,\mu) \cap L^\infty(\Omega,\mu)$, with $\|\phi \|_{2,\mu}=1$.
Then, since $E_S$ is self-adjoint and by Parseval's equality \eqref{parsevalT}, we have  for all $f,g\in L^2(\Omega,\mu)$,
  \begin{eqnarray*}
    \scal{L_{S,\Sigma}f,g}_{\mu}  &=&   \scal{F_\Sigma E_Sf , \chi_S\, g}_{\mu}\\
&=& \sqrt{\mu(S)} \scal{F_\Sigma E_Sf , \phi g}_{\mu}\\
&=&\sqrt{ \mu(S)} \scal{\tt F_\Sigma E_Sf , \tt(\phi g)}_{\mu}\\
&=& \sqrt{\mu(S) }\scal{\chi_\Sigma \tt \chi_S f , \tt(\phi g)}_{\mu}\\
&=& \mu(S)  \scal{\sigma \tt (\phi f) , \tt(\phi g)}_{\mu}\\
&=& \mu(S)  \scal{P_{\Sigma,\phi} , g}_{\mu}.\\
      \end{eqnarray*}
This completes the proof.
  \end{proof}
   From Theorem \ref{th.traceclass} and Theorem \ref{FSRO}, we deduce the following trace formula.
 \begin{corollary}\ \label{cor:tracefs}
The phase space operator $L_{S,\Sigma} $ is trace class
 with
 \begin{equation}\label{tracefs}
  \tr\left(L_{S,\Sigma}\right)=   \mu(S) \tr\left(P_{\Sigma,\phi} \right)=  \int_{S } \int_\Sigma  |\kk(x,\xi)|^2 \d\mu(x) \d\mu(\xi).
\end{equation}
\end{corollary}

\medskip

The compact operator $L_{S,\Sigma} :L^2 (\Omega,\mu)\to L^2(\Omega,\mu)$ is self-adjoint and then can be diagonalized as
\begin{equation}
L_{S,\Sigma} f=\sum_{n=1}^{\infty} \lambda_n \scal{f,\varphi_n}_{\mu}\varphi_n ,
\end{equation}
where $\{\lambda_n=\lambda_n(S,\Sigma)\}_{n=1}^{\infty}$ are the positive eigenvalues associated to the
 arranged in a non-increasing manner
\begin{equation}
 \lambda_n \le \cdots\le \lambda_1< 1 ,
 \end{equation}
 and  $\{\varphi_n =\varphi_n (S,\Sigma)\}_{n=1}^{\infty}$ is the corresponding orthonormal set of eigenfunctions.
In particular
\begin{equation}
  \|L_{S,\Sigma} \|_{S_\infty}= \lambda_1,
\end{equation}
where  $\lambda_1$ is the first eigenvalue corresponding  to the first eigenfunction $\varphi_1$ of the compact operator $L_{S,\Sigma}  $.
This eigenfunction   realizes the maximum of concentration on the set $S\times \Sigma$.
On the other hand, since  $\varphi_n$ is an eigenfunction of $L_{S,\Sigma}$ with eigenvalue $\lambda_n$, then
\begin{equation}\label{eq.eigen}
\|L_{S,\Sigma}\varphi_n-\varphi_n \|_{2,\mu}=\scal{\varphi_n-L_{S,\Sigma}\varphi_n,\varphi_n}_{\mu}=1-\lambda_n,
\end{equation}
and
\begin{eqnarray}\label{eq.eigenbis}
\|L_{S,\Sigma}\left( L_{S,\Sigma}\varphi_n\right)-L_{S,\Sigma}\varphi_n \|_{2,\mu}
&=&\lambda_n^{-1}\scal{L_{S,\Sigma}\varphi_n-L_{S,\Sigma}\left( L_{S,\Sigma}\varphi_n\right), L_{S,\Sigma}\varphi_n}_{\mu}\nonumber\\
&=&\lambda_n(1-\lambda_n)=(1-\lambda_n)\|L_{S,\Sigma} \varphi_n \|_{2,\mu}.
\end{eqnarray}
Thus, for all  $n $, the functions $\varphi_n$ and $L_{S,\Sigma}\varphi_n$ are $(1-\lambda_n)$-localized  with respect to $L_{S,\Sigma}$.
More generally,  we have the following comparisons of the measures of localization.
\begin{proposition}\ \label{prop.comparison}
Let $\eps,\eps_1,\eps_2\in(0,1).$
\begin{enumerate}
 \item  If $f\in L^2(\eps_1,\eps_2, S, \Sigma)$, then $f$ is $(\eps_1+\eps_2)$-localized  with respect to $ F_\Sigma E_S $ and
$(2\eps_1+\eps_2)$-localized  with respect to $L_{S,\Sigma}$.
\item  If $f\in L^2(\Omega,\mu)$ is $\eps$-localized with respect to $L_{S,\Sigma}$, then
\begin{equation}
   \scal{f-L_{S,\Sigma}f,f}_\mu\le (\eps^2+\eps) \|f\|_{2,\mu}^2.
\end{equation}

\item  If $f \in L^2(\Omega,\mu)$ satisfies
\begin{equation}\label{eqnewdef}
\scal{f-L_{S,\Sigma}f,f}_\mu\le \eps \|f\|_{2,\mu}^2,
\end{equation}
then $f$ is   $\sqrt{\eps}$-localized with respect to $L_{S,\Sigma}$.

    \item  If $f\in L^2(\eps_1,\eps_2, S, \Sigma)$, then
    \begin{equation} \label{eqap1ep2L}
   \scal{f-L_{S,\Sigma}f,f}_\mu< (2\eps_1+\eps_2) \|f\|_{2,\mu}^2.
\end{equation}
\end{enumerate}
\end{proposition}
\begin{proof}
Recall that $\|E_S\|_{S_\infty}=\|F_\Sigma\|_{S_\infty}=1$. First we have
\begin{eqnarray*}
   \| F_\Sigma E_S f-f\|_{2,\mu} &\le&  \|  F_\Sigma f-f \|_{2,\mu} +\|  F_\Sigma E_S  f-F_\Sigma f \|_{2,\mu}  \\
     &\le&  \|F_{\Sigma^c} f\|_{2,\mu} +\|F_\Sigma\|_{S_\infty}\| E_{S^c}f\|_{2,\mu}\\
     &\le&   (\eps_1 + \eps_2 )\|f\|_{2,\mu}.
 \end{eqnarray*}
Moreover,
 \begin{eqnarray*}
   \| L_{S,\Sigma} f-f\|_{2,\mu} &\le& \|  E_S F_\Sigma E_S f-   E_S f \|_{2,\mu}+ \|   E_S f-  f\|_{2,\mu}  \\
     &\le&  \|E_S\|_{S_\infty} \|   F_\Sigma E_S f-    f \|_{2,\mu}+ \|   E_S f-  f\|_{2,\mu} \\
    &\le&   (2\eps_1 + \eps_2)\|  f\|_{2,\mu} .
 \end{eqnarray*}
Now since
   \begin{eqnarray*}
     2\scal{f-L_{S,\Sigma} f,f}_{\mu} &=& \norm{L_{S,\Sigma}f- f}_{2,\mu}^{2}+\norm{f}_{2,\mu}^{2}-\norm{L_{S,\Sigma}f }_{2,\mu}^{2} \\
       &\le&    \norm{L_{S,\Sigma}f- f}_{2,\mu}^{2}+\left(\norm{L_{S,\Sigma}f-f}_{2,\mu} +\norm{L_{S,\Sigma}f }_{2,\mu} \right)^{2}
      -\norm{L_{S,\Sigma}f }_{2,\mu}^{2} \\
        &=& 2\norm{L_{S,\Sigma}f- f}_{2,\mu}^{2}+2\norm{L_{S,\Sigma}f-f}_{2,\mu}\norm{L_{S,\Sigma}f}_{2,\mu},
   \end{eqnarray*}
and since $\norm{L_{S,\Sigma} }_{S_\infty}\le 1$, then
   \begin{equation}
     \scal{f-L_{S,\Sigma}f,f}_{\mu}\le \norm{L_{S,\Sigma}f- f}_{2,\mu}^{2}+ \norm{L_{S,\Sigma}f-f}_{2,\mu}\norm{f}_{2,\mu}\le(\eps^2+\eps)\norm{f}_{2,\mu}^2,
   \end{equation}
    and the second result follows.

 On the other hand, since
   \begin{equation}
      \scal{\left(L_{S,\Sigma}\right)^{2}f,f}_{\mu}  \le  \scal{L_{S,\Sigma}f,f}_{\mu}    ,
   \end{equation}
   and since $L_{S,\Sigma}$ is self-adjoint, then
    \begin{equation}
    \norm{L_{S,\Sigma}f- f}_{2,\mu}^{2}=  \scal{\left(I-L_{S,\Sigma}\right)^{2}f,f}_{\mu}  \le  \scal{(I-L_{S,\Sigma})f,f}_{\mu}\le \eps \|f\|_{2,\mu}^2  .
   \end{equation}
   Finally, since
 $$
\scal{f-L_{S,\Sigma}f, f}_\mu  =     \scal{E_{S^c} f, f}_\mu  + \scal{E_{S} f , F_{\Sigma^c}f }_\mu+\scal{F_\Sigma E_S f, E_{S^c}f}_\mu,
$$
then we obtain the last result.
   \end{proof}

\bigskip

The definition \eqref{eqnewdef} is equivalent to
\begin{equation}\label{eqnewdef2}
    \scal{L_{S,\Sigma}f, f}_\mu \ge (1-\eps)\|f\|_{2,\mu}^2  ,
\end{equation}
and we denote by $L^2(\eps,S,\Sigma)$ the subspace of $L^2(\Omega,\mu)$ consisting of functions $f\in L^2(\Omega,\mu)$ satisfying \eqref{eqnewdef2}. Hence from
\eqref{eq.eigen} and \eqref{eq.eigenbis} we have,
 \begin{equation}\label{eq.finlfn}
  \forall \, n\ge1,\quad \varphi_n, \, L_{S,\Sigma}\varphi_n \in L^2(1-\lambda_n,S,\Sigma).
 \end{equation}
Moreover from Proposition \ref{prop.comparison}, if $f\in L^2(\eps_1,\eps_2,S,\Sigma)$, then  $f\in L^2(2\eps_1+\eps_2,S,\Sigma)$, and if $f$ is $\eps$-localized with respect to  $L_{S,\Sigma}$, then $f\in L^2(2\eps,S,\Sigma)$.
Therefore we are interested to study  the following optimization problem
\begin{equation}\label{P}
 \mathrm{Maximize} \quad \quad      \scal{L_{S,\Sigma}f,f}_{ \mu}  , \quad \quad \norm{f}_{2,\mu}= 1,
\end{equation}
which aims to look for    orthonormal functions in     $L^2(\Omega,\mu)$, which are approximately time and band-limited to a bounded region like $S\times\Sigma$.
It follows that the number of eigenfunctions of $L_{S,\Sigma}$ whose eigenvalues are very close to one, are an optimal solutions to
the problem \eqref{P}, since  if $\varphi_n$ is an eigenfunction of $L_{S,\Sigma}$ with eigenvalue $\lambda_n \ge(1 - \eps)$, we have from the spectral representation,
   \begin{equation} \label{eq.eigenge}
     \scal{L_{S,\Sigma}\varphi_n,\varphi_n}_{\mu} = \lambda_n \ge (1 - \eps).
   \end{equation}
  We denote by $n(\eps,S,\Sigma)$ for the number of eigenvalues  $\lambda_n$ of  $L_{S,\Sigma}$ which are close to one, in the sense that
    \begin{equation}
    \lambda_1\ge\cdots \ge   \lambda_{n(\eps,S,\Sigma)}\ge 1-\eps>\lambda_{1+n(\eps,S,\Sigma)} \ge \cdots,
      \end{equation}
and we denote by  $V_{n(\eps,S,\Sigma)}=\mathrm{span}\left\{\varphi_n  \right\}_{n=1}^{n(\eps,S,\Sigma)}$  the span of the first  eigenfunctions  of $L_{S,\Sigma}$  corresponding to the largest eigenvalues $\left\{\lambda_n  \right\}_{n=1}^{n(\eps,S,\Sigma)}$.
 Therefore, by \eqref{eq.eigenge}  and  \eqref{eq.finlfn},  each eigenfunction $\varphi_n $
and its resulting function $L_{S,\Sigma}\varphi_n$ are in $L^2(\eps,S,\Sigma)$, if and only if $1\le n\le n(\eps,S,\Sigma)$.
Now,  if $f \in V_{n(\eps,S,\Sigma)}$, then
  \begin{equation*}\label{eqdefeps3}
  \scal{L^{\psi}_{\Sigma} f,f}_{\mu}   =\sum_{n=1}^{n(\eps,S,\Sigma)}\lambda_n\abs
  {\scal{f,\varphi_n}_{\mu}}^{2} \ge \lambda_{n(\eps,S,\Sigma)}  \sum_{n=1}^{n(\eps,S,\Sigma)} \abs
  {\scal{f,\varphi_n }_{\mu }}^{2} 
  \ge (1- \eps)\|f\|_{2,\mu }^{2}.
\end{equation*}
Thus   $ V_{n(\eps,S,\Sigma)}$  determines the subspace of $L^2 (\Omega,\mu)$ with  maximum dimension   that is   in $ L^2(\eps,S,\Sigma)$.
 Motivated by the recent paper \cite{DAV} in the Gabor setting, we obtain the following theorem that characterizes functions that are  in $ L^2(\eps,S,\Sigma)$.
 \begin{theorem}\
  Let $f_{\mathrm{ker}}$ denote the orthogonal projection of $f$ onto the kernel $\mathrm{Ker}(L_{S,\Sigma})$ of $L_{S,\Sigma}$. Then a function $f $ is in  $ L^2(\eps,S,\Sigma)$ if and only if,
\begin{equation*}\  
    \sum_{n=1}^{n(\eps,S,\Sigma)}(\lambda_n +\eps-1) \abs {\scal{   f,\varphi_n }_{\mu}}^{2} \ge (1 - \eps)\|f_{\mathrm{ker}}\|_{2,\mu}^{2}+
     \sum_{n= 1+n(\eps,S,\Sigma)}^{\infty}(1-\eps-\lambda_n) \abs {\scal{f,\varphi_n }_{\mu }}^{2}.
\end{equation*}
\end{theorem}
\begin{proof}
For a given function $f\in L^2 (\Omega,\mu)$, write
\begin{equation}
   f=\sum_{n=1}^{\infty}\scal{ f,\varphi_n}_{\mu}\varphi_n + f_{\mathrm{ker}},
\end{equation}
where $f_{\mathrm{ker}}\in\mathrm{Ker}(L_{S,\Sigma})$. Then
\begin{equation}\label{eqdefehps33}
  \scal{ L_{S,\Sigma} f,f}_{\mu}   =\sum_{n=1}^{\infty}\lambda_n  \abs {\scal{f,\varphi_n }_{\mu }}^{2} .
\end{equation}
So the function $f$ is in $L^2(\eps,S,\Sigma) $  if and only if
\begin{equation}
  \sum_{n=1}^{\infty}\lambda_n   \abs {\scal{ f,\varphi_n}_{\mu_k}}^{2}
  \ge (1-\eps) \left(\|f_{\mathrm{ker}}\|_{2,\mu}^{2}+\sum_{n=1}^{\infty}\abs {\scal{   f,\varphi_n }_{\mu }}^{2}\right),
\end{equation}
and the conclusion follows.
\end{proof}

 While a function $f$ that is in $L^2(\eps,S,\Sigma) $ does not necessarily lies in some subspace
  $V_{N}=\mathrm{span}\{\varphi_n \}_{n=1}^{N}$, it can be approximated using a finite number of  such eigenfunctions.
Let $\eps_0\in (0, 1)$ be a fixed real number and let   $\pp  $   the orthogonal projection onto the subspace $V_{n(\eps_0,S,\Sigma)}$.
\begin{theorem}\ \label{thaproximtion}
Let $f$ be a function in $L^2(\eps,S,\Sigma) $. Then
\begin{equation}\label{eq.aproximtion}
  \norm{f-\sum_{n=1}^{n(\eps_0,S,\Sigma)}\scal{f,\varphi_n}_\mu\varphi_n}_{2,\mu}\le \sqrt{\frac{\eps}{\eps_0}}\;\|f\|_{2,\mu}.
\end{equation}

\end{theorem}

\begin{proof}
 An easy adaptation of the proof of Proposition 3.3 in \cite{DAV}, we can conclue that
\begin{equation}
 \|\pp  f\|_{2,\mu}^{2}\ge (1- \eps/\eps_0)\|f\|_{2,\mu}^{2}.
\end{equation}
It follows then,
\begin{equation*}
     \|f\|_{2,\mu}^{2}=  \|\pp f + (f - \pp f)\|_{2,\mu }^{2}= \|\pp f\|_{2,\mu }^{2}+\|f-\pp f\|_{2,\mu }^{2} .
\end{equation*}
Thus
\begin{equation*}
    \|f-\pp f\|_{2,\mu}^{2}= \| f\|_{2,\mu}^{2}-\|\pp f\|_{2,\mu}^{2}\le \| f\|_{2,\mu}^{2}-(1- \eps/\eps_0) \|f\|_{2,\mu}^{2}= \eps/\eps_0  \|f\|_{2,\mu}^{2}.
\end{equation*}
 This completes the proof of the theorem.
\end{proof}
Consequently and from Proposition \ref{prop.comparison}, we immediately deduce the following approximating  results.
\begin{corollary}\ Let $\eps,\eps_1,\eps_2\in(0,1)$.
\begin{enumerate}
  \item If $f\in L^2(\eps_1,\eps_2,S,\Sigma) $, then
\begin{equation}\label{eq.aproximtion2}
  \norm{f-\sum_{n=1}^{n(\eps_0,S,\Sigma)}\scal{f,\varphi_n}_\mu\varphi_n}_{2,\mu}\le \sqrt{\frac{2\eps_1+\eps_2}{\eps_0}}\;\|f\|_{2,\mu}.
\end{equation}
  \item If $f\in L^2(\Omega,\mu) $ is $\eps$-localized with respect to $L_{S,\Sigma}$, then
\begin{equation}\label{eq.aproximtion3}
  \norm{f-\sum_{n=1}^{n(\eps_0,S,\Sigma)}\scal{f,\varphi_n}_\mu\varphi_n}_{2,\mu}\le \sqrt{\frac{2\eps}{\eps_0}}\;\|f\|_{2,\mu}.
\end{equation}
\end{enumerate}

\end{corollary}

\subsection{Shapiro--type uncertainty principles}
Based on Malinnikova's ideas \cite{malin}, we will prove in this section a  quantitative dispersion inequality for orthonormal sequences and a  strong
uncertainty principle for orthonormal bases.
Notice also that the homogeneity of the kernel $\kk$ plays a key role in this section, especially in Lemma \ref{lem1}.

\subsubsection{Localization theorem}
Our starting point is the following theorem which states that any orthonormal system in $L^2 (\eps_1,\eps_2,S,\Sigma)$ cannot be infinite.
\begin{theorem}\label{th1}
Let $\eps_1,\eps_2\in(0,1)$ such that $2\eps_1+\eps_2<1$ and let $\left\{f_n \right\}_{n=1}^N$ be an orthonormal system in $L^2 (\eps_1,\eps_2,S,\Sigma)$.
Then
\begin{equation}\label{eqth1}
 N < c_\kk^2 \frac{ \mu(S)\mu(\Sigma)}{1- 2\eps_1-\eps_2 }.
\end{equation}
\end{theorem}
\begin{proof}
The result follows immediately from Inequalities \eqref{eqap1ep2L} and \eqref{tracefs}.
%
 \end{proof}
Consequently, if the generalized dispersions of each element  of an orthonormal sequence are uniformly bounded, then this sequence is also finite.

 \begin{corollary} \label{cor2}
Fix $A_1, A_2>0$. Let $s>0$  and let $\left\{f_n \right\}_{n=1}^N$ be an orthonormal sequence  in $  L^2(\Omega,\mu)$  such that
 $\big\||x|^{s}f_n\big\|^{1/s}_{2,\mu}\le A_1$ and $\big\||\xi|^{s}\,\tt f_n\big\|^{1/s}_{2,\mu}\le A_2$.  Then each $f_n$ is in $ L^2\left(\frac{1}{4},\frac{1}{4}, B_{4^{\frac{1}{s}}A_1}, B_{4^{\frac{1}{s}}A_2}  \right)$, and  
\begin{equation}
N\le C\,(A_1A_2)^{2a}.
\end{equation}
\end{corollary}

\begin{proof}
By assumption we have, for all $n\ge 1$,
\begin{equation*}
 \int_{|x|> 4^{\frac{1}{s}}A_1} |f_n(x)|^2\d\mu(x) = \int_{|x|> 4^{\frac{1}{s}}A_1} |x|^{-2s}|x|^{2s}|f_n(x)|^2\d\mu(x)
 \le \frac{1}{16A_1^{2s}}\big\||x|^{s}f_n\big\|^2_{2,\mu}\le \frac{1}{16}.
\end{equation*}
In the same way we obtain
\begin{equation*}
 \int_{|\xi|> 4^{\frac{1}{s}}A_2} |\tt f_n(\xi)|^2\d\mu(\xi) \le \frac{1}{16}.
\end{equation*}
Thus $f_n \in L^2\left(\frac{1}{4},\frac{1}{4}, B_{4^{\frac{1}{s}}A_1}, B_{4^{\frac{1}{s}}A_2}  \right)$, 
and from    \eqref{eqth1} we conclude the desired result.
\end{proof}

\subsubsection{Quantitative   dispersion inequality for orthonormal sequences}
From Inequality \eqref{eqheisenbergTint},  there exists a constant $C$ such that for all $f\in  L^2(\Omega,\mu)$,
  \begin{equation} \label{varn0}
     \||x|^{s}f \|_{2,\mu} \;  \||\xi|^{s}\tt f  \|_{2,\mu}\ge C \norm{f}^{2}_{2,\mu},
  \end{equation}

and    the dilation argument \eqref{dt} shows that \eqref{varn0}    is equivalent to
\begin{equation}\label{varn0biss}
     \||x|^{s}f \|_{2,\mu}^2 +  \||\xi|^{s}\tt f \|_{2,\mu}^2 \ge 2C \norm{f}^{2}_{2,\mu}.
  \end{equation}
Consequently  we   obtain immediately the following result.
\begin{corollary}\label{corexisj}
Let $s>0$ and let $\left\{f_n \right\}_{n=1}^{\infty}$ be an orthonormal sequence in $L^2(\Omega,\mu)$. Then there exists $j_0\in \Z$ such that
\begin{equation}\label{eqcorexisj}
 \forall n\ge1, \ \ \ \ \ \ \ \  \max\left( \big\||x|^{s}f_n\big\|_{2,\mu}, \big\||\xi|^{s}\tt f_n \big\|_{2,\mu} \right)\ge 2^{s(j_0-1)}.
\end{equation}
\end{corollary}
This corollary with Corollary \ref{cor2} allows as to prove the following quantitative dispersion inequality.

\begin{theorem}
Let $s>0$ and let $\left\{f_n \right\}_{n=1}^{\infty}$ be an orthonormal sequence in $L^2(\Omega,\mu)$.
Then for every  $N\ge 1$, 
\begin{equation}
    \dst \sum_{n=1}^N  \left(\big\||x|^{s}f_n\big\|^{2}_{2,\mu}+\big\||\xi|^{s}\tt f_n \big\|^{2}_{2,\mu}\right)
    \ge C \,  N^{1+\frac{s}{2a}}.
\end{equation}
\end{theorem}

\begin{proof}
For each $j\in\Z$ we define
$$
P_j=\left\{n\tq \max\left( \big\||x|^{s}f_n\big\|^{\frac{1}{s}}_{2,\mu}, \big\||\xi|^{s}\tt f_n \big\|^{\frac{1}{s}}_{2,\mu} \right)\in\left[2^{j-1},2^j\right)\right\}.
$$
First, by Inequality \eqref{eqcorexisj}, we see that  $P_j$ is empty for all $j<j_0$.
Moreover, since for each $ n\in P_j$,  $j\ge j_0$,
\begin{equation}\label{eq:and}
\big\||x|^{s}f_n\big\|_{2,\mu}^{\frac{1}{s}} \le  2^{j} \ \ \ \ \mathrm{and}     \ \ \ \
 \big\||\xi|^{s}\tt f_n \big\|_{2,\mu}^{\frac{1}{s}} \le  2^{j},
\end{equation}
then by Corollary \ref{cor2},  $P_j$ is finite, for all $j\ge j_0$, and if we denote by $N_j$   the number of elements in $P_j$ then
\begin{equation*}
 N_j\le C  \, 4^{2aj}.
\end{equation*}
Therefore, for every $m\ge j_0$, the number of elements in $\bigcup_{j=j_0}^mP_j$ is less than
$C \,   4^{2a m },$
where $C $ 
 is a constant that does not depend on $m$.

Now, if $N>2C \,   4^{2aj_0}$, then we can   choose an  integer $m> j_0$ such that
$$ 2 C \,    4^{(m-1) 2a}<N\le 2C \,    4^{2am }.$$
 Therefore  at least half of $\{1,\ldots, N\}$ does
not belong to $\bigcup_{j=j_0}^{m-1}P_j$ and we obtain
\begin{eqnarray*}
        \sum_{n=1}^N\left(\big\||x|^{s}f_n\big\|^{2}_{2,\mu}+\big\||\xi|^{s}\tt f_n \big\|^{2}_{2,\mu}\right)
 &\ge& \sum_{n=1}^N\max\left(\big\||x|^{s}f_n\big\|^{2}_{2,\mu},\big\||\xi|^{s}\tt f_n \big\|^{2}_{2,\mu}\right) \\
 &\ge& \frac{N}{2} 4^{s(m-1)}
   \ge  \frac{1}{2} \frac{N}{4^{s}}  \left(\frac{N}{2C  }\right)^{\frac{s}{2a}} .
\end{eqnarray*}

Finally, if $N\le2C \,  4^{2aj_0 }$, then from Corollary \ref{corexisj} we have 
\begin{eqnarray*}
          \sum_{n=1}^N\left(\big\||x|^{s}f_n\big\|^{2}_{2,\mu}+\big\||\xi|^{s}\tt f_n \big\|^{2}_{2,\mu}\right)
 &\ge& \sum_{n=1}^N\max\left(\big\||x|^{s}f_n\big\|^{2}_{2,\mu},\big\||\xi|^{s}\tt f_n \big\|^{2}_{2,\mu}\right) \\
 &\ge& N 4^{s(j_0-1)}
   \ge  \frac{N}{4^{s}}  \left(\frac{N}{2C}\right)^{\frac{s}{2a}}.
\end{eqnarray*}
This completes the proof.
\end{proof}
The last Dispersion inequality implies in particular that, there does not exist an infinite sequence $\left\{f_n \right\}_{n=1}^{\infty}$  in $L^2(\Omega,\mu)$
such that the two sequences  $\left\{\||x|^s f_n \|_{2,\mu} \right\}_{n=1}^\infty$ and
$\left\{\||\xi|^s  \tt f_n \|_{2,\mu} \right\}_{n=1}^\infty$ are  bounded. More precisely:
\begin{corollary}
Let $s>0$ and let $\left\{f_n \right\}_{n=1}^{\infty}$ be an orthonormal sequence in $L^2(\Omega,\mu)$.
  Then for every $N\ge1$,
\begin{equation}
\sup_{1\le n\le N}\left\{\big\||x|^s f_n\big\|_{2,\mu}^2,\; \big\||\xi|^s \tt f_n \big\|_{2,\mu}^2\right\}
 \ge C\, N^{\frac{s}{2a}}.
\end{equation}
In particular
\begin{equation}\label{eq.stronsum}
\sup_{n}\left(\big\||x|^s f_n\big\|_{2,\mu}^2 + \big\||\xi|^s \tt f_n \big\|_{2,\mu}^2\right)=\infty.
\end{equation}
\end{corollary}



\subsubsection{Strong uncertainty principle for orthonormal bases} One wonders if \eqref{eq.stronsum} still valid for the product instead of the sum. We will show that this statement is not true in general for orthonormal sequences, but still valid for orthonormal bases.
Indeed, there is an infinite orthonormal sequence $\left\{f_n \right\}_{n=1}^{\infty}$ in $L^2(\Omega,\mu)$,
 with  bounded product of dispersions. Fix $f:\Omega\to \R$  a radial, real-valued Schwartz function supported in $B(0,2)\backslash B(0,1)$, with $\|f\|_{2,\mu}=1$, and consider $f_n(x)=2^{na}f(2^nx)$. Then
$$
\|f_n\|_{2,\mu}=\|f\|_{2,\mu}, \ \ \ \  \supp f_n\subset  B(0,2^{-n+1})\backslash B(0,2^{-n}) \ \ \ \ \mathrm{and} \ \ \ \ \tt f_n (\xi)=2^{-na}\tt(f)\left(2^{-n}\xi\right).
$$
Therefore $\left\{f_n \right\}_{n=1}^{\infty}$ form an orthonormal sequence in $L^2(\Omega,\mu)$ and for every $s>0$,
$$
\big\||x|^{s} f_n\big\|_{2,\mu} =2^{-ns}\big\||x|^{s} f\big\|_{2,\mu}, \ \ \ \ \ \ \  \big\||\xi|^{s} \tt f_n \big\|_{2, \mu} =2^{ns} \big\||\xi|^{s} \tt f \big\|_{2, \mu}.
$$
Hence for all $n$,
 $$
 \big\||x|^{s} f_n\big\|_{2,\mu} \, \big\||\xi|^{s }\tt f_n \big\|_{2, \mu }=\big\||x|^{s } f\big\|_{2,\mu}\big\||\xi|^{s} \tt f \big\|_{2,\mu}<\infty.
 $$

To prove the main result of this subsection, we will need the following special form of the uncertainty principle for sets of finite measure, see \eg \cite{AB, Be, GJ, GJstudia, MZ}.
\begin{lemma} \label{lem1}
Let  $S$ and $\Sigma$ be measurable subsets of finite  measure  $ 0<\mu(S),\,\mu(\Sigma)<\infty$. Then there exists a nonzero function $f\in L^2(\Omega,\mu)$ such that
$\supp f\subset S^c$ and $\supp \tt  f \subset \Sigma^c$.
\end{lemma}

\begin{proof}
From \cite[Corollary 3.7]{GJstudia}, there exist a positive constant $C(S,\Sigma)$ such that
for all functions $f\in \Im(F_\Sigma)$,
$$\|f\|_{2,\mu}\leq C (S,\Sigma)\|E_{S^c}f\|_{ 2,\mu}.$$
Therefore  the trace space $\Lambda=\{ f|_{S^c }: f\in\Im(F_\Sigma) \}$   form a closed subspace in $ L^2 (S^c,\mu)  $ which is obviously not the whole space.
Let $g$ be a nonzero function in $\Lambda^c=  L^2(S^c,\mu)\backslash  \Lambda$. Since $g=F_{\Sigma}g+F_{\Sigma^c}g$, then $f=F_{\Sigma^c}g$
is a nonzero function  in $ L^2(\Omega,\mu)$ such that $f$ is supported on $S^c$ and $\tt f $ is supported on $\Sigma^{c}$.
We extend $f$ by zero on $S$ in order to get the required function.
\end{proof}

\begin{theorem}
Let $s>0$ and let $\left\{f_n \right\}_{n=1}^{ \infty}$ be an orthonormal basis for $ L^2(\Omega,\mu)$. Then
\begin{equation}
\sup_n \left(\big\||x|^{s}f_n\big\|_{2,\mu} \,\big\||\xi|^{s}\tt f_n \big\|_{ 2,\mu }\right)= \infty.
\end{equation}
\end{theorem}
\begin{proof}
Assume that there exists an orthonormal basis $\left\{f_n \right\}_{n=1}^{ \infty}$ such that
 $$\big\||x|^{s}f_n\big\|_{2,\mu}^{1/s} \big\||\xi|^{s}\tt f_n \big\|_{2,\mu}^{1/s}\le A^2 .$$
Let $k\in \Z$ and let
$$
A_k=\left\{f_n \tq  \big\||x|^{s}f_n\big\|_{2,\mu}^{1/s}\in \left(2^{-k}A, 2^{-k+1}A\right] \right\}.
$$
Clearly,  $ \left\{f_n \right\}_{n=1}^{ \infty}= \bigcup_k A_k$, and for each $f_n\in A_k$, we have $$\big\||x|^{s}f_n\big\|_{2,\mu}^{1/s}\le  2^{-k+1}A \ \ \ \  \mathrm{and}   \ \ \ \
  \big\||\xi|^{s}\tt f_n \big\|_{2,\mu}^{1/s}\le A2^k.$$ Then by Corollary \ref{cor2},
  $A_k$ is finite, and if $N_k$ is the number of elements in $A_k$
then $N_k$ is bounded by a constant $C$ that does not depend on $k$.

Let $R>0$, then by using  Lemma \ref{lem1}, we take a nonzero function  $f\in L^2(\Omega,\mu)$ with $\|f\|_{2,\mu} =1$,
such that  $$\supp f,\, \supp \tt f\subset B_R^c.$$ Then for $k\ge0$ and $f_n\in A_k$ we obtain by the Cauchy-Schwartz inequality that
\begin{equation}\label{eqA}
  \left|\scal{f,f_n}_\mu\right|^2\le R^{-2s}\|f\|_{2,\mu }^{2} \big\||x|^{s} f_n \big\|_{ 2,\mu }^{2}
    \le (2AR^{-1})^{2s}  4^{-sk}.
\end{equation}

Similarly, for $k<0 $ and $f_n\in A_k$ we obtain by Parseval theorem \eqref{parsevalT},
\begin{equation}\label{eqB}
    \left|\scal{f,f_n}_\mu\right|^2 =  \left|\scal{\tt f ,\tt f_n }_\mu\right|^2
     \le R^{-2s}\|f\|_{2,\mu}^{2} \big\||\xi|^{s}\tt f_n \big\|_{2,\mu}^{2}
    \le (AR^{-1})^{2s}  4^{sk}.
\end{equation}
Now, since $\left\{f_n \right\}_{n=1}^{ \infty}$ is an orthonormal basis for $L^2(\Omega,\mu)$, then
$$
1=\|f\|_{2,\mu}^{2}= \dst \sum_{k} \sum_{f_n\in A_k}\left|\scal{f,f_n}_\mu\right|^2,
$$
and by combining Inequalities \eqref{eqA} and  \eqref{eqB}, we obtain
\begin{eqnarray*}
  1 &\le& (2AR^{-1})^{2s} \sum_{k=0}^{ \infty}  4^{-sk}N_k +  (AR^{-1})^{2s} \sum_{k=1}^{ \infty}  4^{-sk}N_{-k} \\
    &\le& C(2AR^{-1})^{2s} \sum_{k=0}^{ \infty}  4^{-sk} + C(AR^{-1})^{2s} \sum_{k=1}^{ \infty}  4^{-sk} \\
    &\le& \dfrac{C }{ R^{2s}}.
\end{eqnarray*}
Choosing $R$ large enough, we get a contradiction. The theorem is proved.
\end{proof}

\end{document}